\documentclass[11pt]{article}
\usepackage{amssymb}
\usepackage{amsmath}
\usepackage{amsthm}
\usepackage{latexsym}
\usepackage{hyperref}
\usepackage{mathdots}
\usepackage[all]{xy}
\usepackage{stmaryrd}
\usepackage[normalem]{ulem}
\usepackage[capitalise]{cleveref}
\usepackage{tikz}
\usetikzlibrary{patterns}

\theoremstyle{definition}
\newtheorem{theo}{Theorem}[section]
\newtheorem{defi}[theo]{Definition}
\newtheorem{prop}[theo]{Proposition}

\newtheorem{cor}[theo]{Corollary}
\newtheorem{lemma}[theo]{Lemma}
\newtheorem{exa}[theo]{Example}
\newtheorem{rem}[theo]{Remark}
\numberwithin{equation}{section}

\newcommand{\N}{{\mathbb N}}
\newcommand{\F}{{\mathbb F}}

\newcommand{\cC}{{\mathcal C}}

\newcommand{\cF}{{\mathcal F}}

\newcommand{\cM}{{\mathcal M}}

\newcommand{\cT}{{\mathcal T}}
\newcommand{\cL}{{\mathcal L}}

\newcommand{\cS}{{\mathcal S}}
\newcommand{\cN}{{\mathcal N}}
\newcommand{\cU}{{\mathcal U}}
\newcommand{\cV}{{\mathcal V}}

\newcommand{\qM}{$q$-matroid}
\newcommand{\T}{\mbox{$^{\sf T}$}}
\newcommand{\subspace}[1]{\mbox{$\langle{#1}\rangle$}}

\newcommand{\Aut}{\mbox{\rm Aut}}
\newcommand{\CovFP}{\mbox{${\rm Cov}_{{\mathcal F}'}$}}

\newcommand{\rk}{{\rm rk}}

\newcommand{\im}{\mbox{${\rm im}\,$}}

\newcommand{\GL}{{\rm GL}}
\newcommand{\rowsp}{\mbox{\rm rs}}
\newcommand{\colsp}{\mbox{\rm cs}}
\newcommand{\diag}{\mbox{\rm diag}}
\newcommand{\id}{{\rm id}}

\newcommand{\CC}{\textbf{C}}
\newcommand{\qMat}{$q$-\textbf{Mat}}
\newcommand{\qMats}{$q$-\textbf{Mat}$^{\sf s}$}
\newcommand{\qMatw}{$q$-\textbf{Mat}$^{\sf w}$}
\newcommand{\qMatrp}{$q$-\textbf{Mat}$^{\sf rp}$}
\newcommand{\qMatls}{$q$-\textbf{Mat}$^{\mbox{\scriptsize\sf l-s}}$}
\newcommand{\qMatlw}{$q$-\textbf{Mat}$^{\mbox{\scriptsize\sf l-w}}$}
\newcommand{\qMatlrp}{$q$-\textbf{Mat}$^{\mbox{\scriptsize\sf l-rp}}$}
\newcommand{\qMatd}{$q$-\textbf{Mat}$^{\Delta}$}

\newcounter{alp}
\newcounter{ara}
\newcounter{rom}

\newenvironment{alphalist}{\begin{list}{(\alph{alp})\hfill}{\usecounter{alp}
     \topsep-1.4ex \labelwidth.7cm \leftmargin.7cm \labelsep0cm
     \rightmargin0cm \parsep0ex \itemsep.5ex
     \partopsep-1.4ex}}{\end{list}}
\newenvironment{arabiclist}{\begin{list}{(\arabic{ara})\hfill}{\usecounter{ara}
     \topsep-1ex \labelwidth.7cm \leftmargin.7cm \labelsep0cm
     \rightmargin0cm \parsep0ex \itemsep.5ex
     \partopsep-1ex}}{\end{list}}
\newenvironment{Flist}{\begin{list}{(F\arabic{ara})\hfill}{\usecounter{ara}
     \topsep-1ex \labelwidth.8cm \leftmargin.8cm \labelsep0cm
     \rightmargin0cm \parsep0ex \itemsep.5ex
     \partopsep-1ex}}{\end{list}}

\newenvironment{mylist2}{\begin{list}{(\arabic{ara})\hfill}{\usecounter{ara}
     \topsep-1ex \labelwidth.88cm \leftmargin.88cm \labelsep0cm
     \rightmargin0cm \parsep0ex \itemsep.5ex
     \partopsep-1ex}}{\end{list}}

\makeatletter
\let\@fnsymbol\@arabic
\makeatother
\begin{document}

\title{Coproducts in Categories of $q$-Matroids}
\author{Heide Gluesing-Luerssen\thanks{Department of Mathematics, University of Kentucky, Lexington KY 40506-0027, USA; heide.gl@uky.edu.}\quad and Benjamin Jany\thanks{Department of Mathematics, University of Kentucky, Lexington KY 40506-0027, USA; benjamin.jany@uky.edu.}}

\date{March 8, 2023}
\maketitle
	
\begin{abstract}
\noindent $q$-Matroids form the $q$-analogue of classical matroids.
In this paper we introduce various types of maps between \qM{}s.
These maps are not necessarily linear, but they map subspaces to subspaces and respect the \qM{} structure in certain ways.
The various types of maps give rise to different categories of \qM{}s. We show that only one of these categories possesses a coproduct.
This is the category where the morphisms are linear weak maps, that is,  the rank of the image of any subspace
is not larger than the rank of the subspace itself.
The coproduct in this category is the very recently introduced direct sum of \qM{}s.
\end{abstract}


\section{Introduction}\label{S-Intro}
Due to their close relation to rank-metric codes, \qM{}s have gained a lot of attention in recent years.
They were introduced in \cite{JuPe18}, and their generalization to $q$-polymatroids appeared first in
\cite{GJLR19} and \cite{Shi19}.
While $\F_{q^m}$-linear rank-metric codes induce \qM{}s, $\F_q$-linear rank-metric codes give rise to the
more general $q$-polymatroids.

A \qM{} is the $q$-analogue of a classical matroid:
Its ground space is a finite-dimensional vector space over some finite field $\F_q$, and the rank function is defined on the
lattice of subspaces. Its properties are the natural generalization of the rank function for classical matroids.

Even though the theory of \qM{}s is still in the beginnings, many results have already been derived.
Most importantly, an abundance of cryptomorphisms have been worked out in \cite{BCJ22} (and partly in \cite{JuPe18,BCIJS23}).
They may be regarded as the natural, yet non-trivial, $q$-analogues of the well-known cryptomorphisms for classical matroids.

In this paper we will discuss an aspect of \qM{}s where the theory diverges significantly from that for classical matroids:
maps between \qM{}s and the resulting categories.
Our study of maps and categories is motivated by the question of whether \qM{} constructions can be characterized (or in fact found) via
suitable universal properties; see \cite{HePa18} for the case of classical matroids.
We will see that this is indeed the case for the direct sum.

Just like for classical matroids we define maps between \qM{}s as maps between the ground spaces that respect the \qM{}
structures in a particular way.
We take the most general approach and do not require the maps to be linear.
However, in order to be compatible with \qM{} structures, they have to map subspaces to subspaces.
In other words, they induce maps between the corresponding subspace lattices.
We call such maps $\cL$\emph{-maps}.
As for classical matroids  (see \cite[Ch.~8 and~9]{Wh86}), there are various ways how $\cL$-maps can respect the \qM{} structure.
We call an $\cL$-map \emph{weak} if for every subspace the rank is greater than or equal to the rank of its image, and it is
\emph{rank-preserving} if we have equality for all subspaces.
An $\cL$-map is \emph{strong} if the pre-image of a flat is a flat.
In addition to these properties we will also pay attention to the special case where the $\cL$-maps are linear.
The theory of $\cL$-maps and maps between \qM{}s will be worked out in \cref{S-Maps}.

It is well known that the direct sum of classical matroids is a coproduct in both the category with strong maps as
morphisms and the category with weak maps as morphisms; see \cite[Ex.~8.6, p.~244]{Wh86} (which refers to \cite{CrRo70}).
In \cref{S-NoCopr} we will show that --- in contrast to the matroid case --- the category of \qM{}s with strong maps does not
admit a coproduct, and the same is true for the category with linear strong maps and the category
with weak maps (trivially, the category with linear or non-linear rank-preserving maps does not admit a coproduct).

The main positive result of this paper concerns the direct sum of \qM{}s presented in \cite{CeJu21}.
In \cref{S-lw} we prove that the direct sum is a coproduct in the category with linear weak maps.
The appearance of \cite{CeJu21}, which happened while the first version of this paper was in preparation, is a very
fortunate incidence because the main motivation for our approach has in fact been the
quest for a direct sum.
For \qM{}s this is not a straightforward concept; in particular, none of the many ways of defining the direct sum of
classical matroids has a well-defined $q$-analogue.
Furthermore, the negative results mentioned in the previous paragraph suggest that any meaningful notion of direct sum of \qM{}s
does not behave well with respect to strong maps, i.e., with respect to flats.

The fact that the direct sum in \cite{CeJu21} does behave well with respect to linear weak maps may be interpreted as follows.
For \qM{}s $\cM_i=(E_i,\rho_i),\,i=1,2,$ consider the collection of all \qM{}s with ground space $E_1\oplus E_2$ and whose restriction to~$E_i$ equals~$\cM_i$.
The coproduct property  implies that in this collection the direct sum is the unique \qM{} with the most independent spaces,
that is,  the unique one with the ``least amount of conditions''.
This is indeed a property any notion of direct sum should satisfy.
Therefore, together with the evidence provided in \cite{CeJu21}, our paper contributes to settling on the right notion
of direct sum for \qM{}s.
Further properties of the direct sum have been established in \cite{GLJ23DSCyc}, where,
among other things, it has been shown that every \qM{} decomposes into a direct sum of
indecomposable \qM{}s in a unique way.

Finally, in the short \cref{S-Lclass} we discuss categories of $q$-matroids where the morphisms are the maps between subspace lattices
induced by $\cL$-maps.
It will be shown that this notion of maps between $q$-matroids is too weak to produce meaningful results.

\textbf{Notation:}
Throughout, let $\F=\F_q$ and~$E$ be a finite-dimensional $\F$-vector space.
We denote by $\cL(E)=(\cL(E),\leq,+,\cap)$  the subspace lattice of~$E$.
A $k$-space is a subspace of dimension~$k$.
For a matrix $M\in\F^{a\times b}$ we denote by $\rowsp(M)\in\cL(\F^b)$ and $\colsp(M)\in\cL(\F^a)$ the row space
and column space of~$M$, respectively.
We use the standard notation $[n]$ for the set $\{1,\ldots,n\}$.
For clarity we use  throughout the terminology \emph{classical matroid} for the `non-$q$-version' of matroids,
that is, matroids based on subsets of a finite ground set.

\section{Basic Notions of $q$-Matroids}

\begin{defi}\label{D-qMatroid}
A \emph{$q$-matroid with ground space~$E$} is a pair $(E,\rho)$, where $\rho: \cL(E)\longrightarrow\N_0$ is a map satisfying
\begin{mylist2}
\item[(R1)\hfill] Dimension-Boundedness: $0\leq\rho(V)\leq \dim V$  for all $V\in\cL(E)$;
\item[(R2)\hfill] Monotonicity: $V\leq W\Longrightarrow \rho(V)\leq \rho(W)$  for all $V,W\in\cL(E)$;
\item[(R3)\hfill] Submodularity: $\rho(V+W)+\rho(V\cap W)\leq \rho(V)+\rho(W)$ for all $V,W\in\cL(E)$.
\end{mylist2}
The value $\rho(V)$ is called the \emph{rank} of~$V$ and $\rho(M):=\rho(E)$ is the \emph{rank} of the \qM.
If $\rho$ is the zero map, we call $(E,\rho)$  the \emph{trivial \qM{} on}~$E$.
\end{defi}

A large class of \qM{}s are the representable ones.

\begin{defi}[\mbox{\cite[Sec.~5]{JuPe18}}]\label{D-ReprG}
Consider a field extension $\F_{q^m}$ of~$\F=\F_q$ and let $G\in\F_{q^m}^{k\times n}$ be a matrix of rank~$k$.
Define the map $\rho:\cL(\F^n)\longrightarrow\N_0$ via
\[
    \rho(V)=\rk (GY\T),\ \text{ where $Y\in\F^{y\times n}$ such that $V=\rowsp(Y)$}.
\]
Then~$\rho$ satisfies (R1)--(R3) and thus defines a \qM{} $\cM_G=(\F^n,\rho)$, which has rank~$k$.
We call $\cM_G$ the \qM{} \emph{represented by~$G$}.
Furthermore,  a \qM{} $\cM=(\F^n,\rho)$ of rank~$k$ is \emph{representable over~$\F_{q^m}$} if $\cM=\cM_G$ for some
$G\in\F_{q^m}^{k\times n}$ and it is \emph{representable} if it is representable over $\F_{q^m}$ for some~$m$.
\end{defi}

Not every \qM{} is representable. The first non-representable \qM{}s appeared in \cite[Sec.~4]{GLJ22Gen}.
However, the smallest non-representable \qM{} is the following one, found just recently in~\cite{CeJu21}.

\begin{exa}[\mbox{\cite[Sec.~3.3]{CeJu21}}]\label{E-NonRepr}
Consider $E=\F_2^4$ and let
\[
    \cV=\{\subspace{1000,\,0100},\subspace{0010,\,0001},\subspace{1001,\,0111},\subspace{1011,\,0110}\}
\]
(or any other partial spread of size $4$). Define $\rho(V)=1$ for $V\in\cV$ and $\rho(V)=\min\{2,\dim V\}$ otherwise.
It follows from \cite[Prop.~4.7]{GLJ22Gen} that this does indeed define a \qM{} $\cM=(\F_2^4,\rho)$.
In \cite[Sec.~3.3]{CeJu21} it has been shown that~$\cM$ is a non-representable \qM{}.
In \cite[Sec.~4]{GLJ22Gen} we prove that $\cM$ is not even representable over $\F_2^{4\times m}$  for any~$m$ in the sense of
\cite[Def.~4.1]{GLJ22Gen}.
\end{exa}

We return to general \qM{}s.
In \cite{BCJ22} (and partly in \cite{BCIJS23}) a variety of cryptomorphic definitions of \qM{}s have been derived.
We will need the one based on flats.
In order to formulate the corresponding results, we recall the following notions from lattice theory.
Let $(\cF,\leq)$ be any lattice. We denote by $F_1\vee F_2$ and $F_1\wedge F_2$ the \emph{join} and \emph{meet} of~$F_1$ and~$F_2$,
respectively.
We say that $F_1$ \emph{covers} $F_2$ if $F_2< F_1$ and there exists no $F\in\cF$ such that $F_2<F<F_1$.
The lattice is \emph{semi-modular} if it satisfies: whenever $F_1$ covers $F_1\wedge F_2$, then $F_1\vee F_2$ covers~$F_2$.

\begin{theo}[\mbox{\cite[Prop.~3.8 and Thm.~3.10]{BCIJS23}}]\label{T-AxFlats}
Let $\cM=(E,\rho)$ be a \qM. For $V\in\cL(E)$ define the \emph{closure} of~$V$ as
\[
    \overline{V}=\sum_{   \genfrac{}{}{0pt}{1}{\dim X=1}{\rho(V+X)=\rho(V)}     }\hspace*{-1.2em}X.
\]
$V$ is called a \emph{flat} if $\overline{V}=V$.
Thus~$V$ is a flat if and only if $\rho(V)<\rho(W)$ for every subspace~$W$ strictly containing~$V$.
Denote by $\cF:=\cF_{\rho}$ the collection of all flats of~$\cM$.
Then
\begin{Flist}
\item $E\in\cF$,
\item If $F_1,F_2\in\cF$, then $F_1\cap F_2\in\cF$,
\item For all $F\in\cF$ and all $1$-dimensional subspaces $X\in\cL(E)\setminus\cL(F)$, there exists a unique cover of~$F$ in~$\cF$ containing~$X$.
\end{Flist}
\end{theo}

The converse result tells us that any collection of subspaces satisfying (F1)--(F3) gives rise to a \qM.

\begin{theo}[\mbox{\cite[Cor.~3.11 and Thm.~3.13]{BCIJS23}}]\label{T-FlatsMatroid}
Let~$\cF\subseteq\cL(E)$ be a collection of subspaces satisfying (F1)--(F3) from \cref{T-AxFlats}.
For $V\in\cL(E)$ set $\overline{V}=\bigcap_{   \genfrac{}{}{0pt}{1}{F\in\cF}{V\subseteq F}    }F$.
\begin{alphalist}
\item  $(\cF,\subseteq)$ is a semi-modular lattice with the meet and join of two flats given by
$F_1\wedge F_2=F_1\cap F_2$ and $F_1\vee F_2=\overline{F_1+F_2}$, respectively.
\item  By semi-modularity all maximal chains between any two fixed elements in~$\cF$ have the same length
         (Jordan-Dedekind chain condition).
\item For any $F\in\cF$ denote by $\text{h}(F)$ the length of a maximal chain from~$0_{\cF}$ to $F$.
         Define the map $\rho_{\cF}:\cL(E)\longrightarrow\N_0,\ V\longmapsto \text{h}(\overline{V})$. Then
        $(E,\rho_{\cF})$ is a \qM.
\end{alphalist}
\end{theo}

The above two processes are mutually inverse in the following sense.

\begin{theo}[\mbox{\cite[Thm.~3.13]{BCIJS23}}]\label{T-FlatsRank}\
\begin{alphalist}
\item Let $(E,\rho)$ be a \qM. Then $\rho_{\cF_{\rho}}=\rho$.
\item Let $\cF$ be a collection of subspaces satisfying (F1)--(F3). Then $\cF_{\rho_{\cF}}=\cF$.
\end{alphalist}
\end{theo}

This last result allows us to define \qM{}s based on their flats, denoted by $(E,\cF)$, or based on the rank function, denoted by $(E,\rho)$.

\begin{exa}\label{E-Uniform}
Let $0\leq k\leq \dim E$. We denote by $\cU_k(E)$ the uniform \qM{} of rank~$k$ on~$E$, that is, its rank function is given by
$\rho(V)=\min\{k,\dim V\}$. The flats are given by $\cF(\cU_{k}(E))=\{V\leq E\mid \dim V\leq k-1\}\cup\{E\}$.
Note that $\cU_0(E)$ is the trivial \qM{} on~$E$.
It is well known (\cite[Ex.~31]{JuPe18} or \cite[Cor.~6.6]{GJLR19}) that $\cU_{k}(\F_q^n)$ is representable over any
$\F_{q^m}$ with $m\geq n$: choose any $G\in\F_{q^m}^{k\times n}$ that generates an $\F_{q^m}$-linear MRD code (of rank distance
$n-k+1$).
\end{exa}

\section{Maps Between $q$-Matroids}\label{S-Maps}
In this section we introduce maps between \qM{}s.
The candidates for such maps are (possibly nonlinear) maps between the ground spaces of the \qM{}s with the
property that they map subspaces to subspaces.
Such maps will be called $\cL$-maps.
By definition, $\cL$-maps induce maps between the associated subspace lattices.
As a consequence, one may choose $\cL$-maps or their induced maps as maps between \qM{}s.
In \cref{S-Lclass} we will briefly discuss the second option, while for now we focus on $\cL$-maps
themselves.
As maps between \qM{}s, they should respect the \qM{} structure.
This can be achieved in various ways, and we will introduce the options later in this section.

Throughout this section let $E_1,E_2$ be finite-dimensional $\F$-vector spaces.

\begin{defi}\label{D-MatMap}
Let $\phi:E_1\longrightarrow E_2$ be a map.
We call~$\phi$ an \emph{$\cL$-map} if $\phi(V)\in\cL(E_2)$ for all $V\in\cL(E_1)$.
The induced map from~$\cL(E_1)$ to $\cL(E_2)$ is denoted by~$\phi_\cL$.
A bijective $\cL$-map is called an \emph{$\cL$-isomorphism}.
Finally, $\cL$-maps $\phi,\,\psi$ from~$E_1$ to~$E_2$ are $\cL$-\emph{equivalent}, denoted by $\phi\sim_{\cL}\psi$, if $\phi_\cL=\psi_\cL$.
\end{defi}

An $\cL$-map $\phi:E_1\longrightarrow E_2$ is thus a possibly nonlinear map that maps
subspaces of~$E_1$ to subspaces of~$E_2$.
It satisfies $\phi(0)=0$ and
\begin{equation}\label{e-phiV}
     \phi(\subspace{v})=\subspace{\phi(v)}\ \text{ for all }\ v\in E_1.
\end{equation}
This follows from the fact that $\phi(\subspace{v})$ is a subspace of cardinality at most~$q$ containing $0$ and~$\phi(v)$.
Our definition of $\cL$-isomorphisms is justified by the following simple fact.

\begin{rem}\label{R-SIso}
Let $\phi:E_1\longrightarrow E_2$ be a bijective $\cL$-map. Then~$\phi^{-1}$ is also an $\cL$-map.
To see this, note that $\dim E_1=\dim E_2$ by bijectivity of~$\phi$ and thus the subspace lattices
$\cL(E_1)$ and $\cL(E_2)$ are isomorphic.
Hence $\phi^{-1}$ is also an $\cL$-map.
\end{rem}

Recall that a map $\phi:E_1\longrightarrow E_2$ is $\F$-semilinear if $\phi$ is additive and there exists $\sigma\in\Aut(\F)$ such that
$\phi(c v)=\sigma(c)\phi(v)$ for all $v\in E_1$ and $c\in\F$.
Clearly, any semi-linear map $\phi:E_1\longrightarrow E_2$ is an $\cL$-map.
Here are examples of non-semi-linear $\cL$-maps and of $\cL$-equivalent maps.
A more general construction of $\cL$-equivalent $\cL$-maps will be given in \cref{P-S1dim}(b).

\begin{exa}\label{E-qMMap}
\begin{alphalist}
\item Let $X\lneq E_1$ be any subspace and let $v_1,\ldots,v_{\ell}\in E_1$ be such that
         $\subspace{v_1},\ldots,\subspace{v_\ell}$ are the distinct lines in~$E_1$ that are not contained in~$X$.
         Choose $z\in E_2\setminus\{0\}$. Set
        \[
            \phi:E_1\longrightarrow E_2,\quad v\longmapsto\left\{\begin{array}{cl} 0,&\text{if }v\in X,\\[.5ex] \lambda z,&\text{if $v=\lambda v_i$ for some $i\in[\ell]$}.\end{array}\right.
        \]
        Then $\phi(V)=\{0\}$ for all $V\in\cL(X)$ and $\phi(V)=\subspace{z}$ for all $V\in\cL(E_1)\setminus\cL(X)$.
        Thus,~$\phi$ is an $\cL$-map.
         Furthermore, the pre-image of any subspace is a subspace.
        Indeed,  let $W\leq E_2$. Then $\phi^{-1}(W)=X$ if $z\not\in W$ and $\phi^{-1}(W)=E_1$ if $z\in W$.
        Note that~$\phi$ depends on the choice of the representatives~$v_i$ for the distinct lines.
        One can easily create examples where~$\phi$ is not semi-linear.
        (In fact, one can show that there always exist choices such that~$\phi$ is not semi-linear
        unless $\dim E_1=1$ or $[\dim  X=\dim E_1-1$ and $\F_q=\F_2]$.)
\item  Let $\phi:\F_2^3\longrightarrow\F_2^2$ be given by $\phi(v_1,v_2,0)=(v_1,v_2)$ and $\phi(v_1,v_2,1)=(0,0)$ for all $v_1,v_2\in\F_2$.
        Then $\phi$ is a nonlinear $\cL$-map. In this case, the pre-image of a subspace is not necessarily a subspace, for instance
        $\phi^{-1}(\{(0,0)\})$.
\item Clearly, if $\phi:E_1\longrightarrow E_2$ is an $\cL$-map, and $\psi=\lambda\phi$ for some $\lambda\in\F^*$, then~$\phi$ and~$\psi$ are
        $\cL$-equivalent $\cL$-maps.
\item Let $\F_4=\{0,1,\alpha,\alpha^2\}$ and consider the semi-linear map $\phi:\F_4^2\longrightarrow\F_4^2,\ (v_1,v_2)\longmapsto(v_1^2,v_2^2)$.
        Furthermore, let
        \[
            \psi:\F_4^2\longrightarrow\F_4^2,\quad (v_1,v_2)\longmapsto\left\{\begin{array}{cl}
                \alpha(v_1^2,v_2^2),&\text{if }(v_1,v_2)\in\subspace{(1,1)},\\[.5ex]
                (v_1^2,v_2^2),&\text{otherwise.}\end{array}\right.
        \]
        One easily verifies that~$\psi$ is an $\cL$-map and $\phi$ and~$\psi$ are $\cL$-equivalent.
\end{alphalist}
\end{exa}

We now discuss the relation between $\cL$-maps from~$E_1$ to~$E_2$  and lattice homomorphism from
$(\cL(E_1),\leq,+,\cap)$ to $(\cL(E_2),\leq,+,\cap)$.
To do so, recall the following notions and simple facts.
For further details see for instance \cite[Ch.~3]{Rom08} or \cite[Ch.~2]{DaPr02}.

\begin{rem}\label{R-LattHomo}
Let $(\cL_i,\leq_i,\vee_i,\wedge_i),\,i=1,2,$ be two lattices and  $\phi:\cL_1\longrightarrow\cL_2$ be a map.
Then~$\phi$ is a \emph{lattice homomorphism} if it is meet- and join-preserving, that is,
$\phi(a\wedge_1 b)=\phi(a)\wedge_2\phi(b)$ and $\phi(a\vee_1 b)=\phi(a)\vee_2\phi(b)$ for all $a,b\in\cL_1$.
If each~$\cL_i$ has a least element $0_i$ and a greatest element~$1_i$, we call
$\phi$ a $\{0,1\}$-\emph{lattice homomorphism} if it is a lattice homomorphism satisfying $\phi(0_1)=0_2$ and
$\phi(1_1)=1_2$.
Finally, we call $\phi$ \emph{order-preserving} if $a\leq_1 b$ implies $\phi(a)\leq_2\phi(b)$ for all $a,b\in\cL_1$.
We have the following facts.
\begin{alphalist}
\item Every lattice homomorphism is order-preserving \cite[Prop.~2.19]{DaPr02}. Clearly, every $\cL$-map~$\phi$ induces an order-preserving map~$\phi_{\cL}$.
\item Let $\phi:E_1\longrightarrow E_2$ be an $\cL$-map and $\phi(v_1)=\phi(v_2)\neq0$ for some linearly independent vectors $v_1,v_2\in E_1$.
        Then~$\phi_{\cL}$ is not meet-preserving and thus not a lattice homomorphism.
        An example of such a map $\phi$ is in \cref{E-qMMap}(a) if $\dim E_1>1$.
\item A lattice homomorphism need not be a $\{0,1\}$-lattice homomorphism.
        Consider for instance the embedding $\tau:\F_2^2\longrightarrow\F_2^3,\,(x,y)\longmapsto (x,y,0)$, and the map
        $\phi:\cL(\F_2^2)\longrightarrow\cL(\F_2^3)$ given by $\phi(V)=\tau(V)$.
\item A lattice isomorphism is a $\{0,1\}$-lattice isomorphism.
\item A $\{0,1\}$-lattice homomorphism need not be a lattice isomorphism. Consider, for instance, $\phi:\cL(\F_2^2)\longrightarrow\cL(\F_2^4)$,
        which maps $0$ to~$0$ and $\F_2^2$ to $\F_2^4$ and the three 1-spaces in $\F_2^2$ to the subspaces $\subspace{1000,0100},\,\subspace{0010,0001}$, and $\subspace{1010,0101}$.
\end{alphalist}
\end{rem}

Note that the map in \cref{R-LattHomo}(e) is not induced by an $\cL$-map (because $\dim V<\dim \phi(V)$).
For maps induced by $\cL$-maps we have some stronger statements.

\begin{prop}\label{P-LattHom}
Let $\phi:E_1\longrightarrow E_2$ be an $\cL$-map and $\phi_\cL:\cL(E_1)\longrightarrow\cL(E_2)$ be the induced map.
\begin{alphalist}
\item If~$\phi$ is injective, then $\phi_\cL$ is a lattice homomorphism.
\item Suppose $E_2\neq0$. Then $\phi$ is an $\cL$-isomorphism $\Longleftrightarrow \phi_\cL$ is a $\{0,1\}$-lattice homomorphism
        $\Longleftrightarrow\phi_\cL$ is a lattice isomorphism.
\end{alphalist}
\end{prop}

Note that if $E_2=0$, then the zero map induces a $\{0,1\}$-lattice homomorphism; thus the exceptional case in (b).

\begin{proof}
(a) Let $V_1,V_2\in\cL(E_1)$ and $\hat{V}_i=\phi(V_i)$. Then clearly $\phi(V_1\cap V_2)\subseteq \hat{V_1}\cap\hat{V_2}$. For the converse containment let $\hat{v}\in\hat{V_1}\cap\hat{V_2}$.
Then there exist $v_i\in V_i$ such that $\phi(v_1)=\hat{v}=\phi(v_2)$ and injectivity implies $v_1=v_2\in V_1\cap V_2$.
This shows $\phi(V_1\cap V_2)= \hat{V_1}\cap\hat{V_2}$ and thus $\phi_\cL$ is meet-preserving.
Next, we clearly have $\hat{V_1}+\hat{V_2}\subseteq\phi(V_1+ V_2)$ and equality follows from
$\dim(\hat{V_1}+\hat{V_2})=\dim\hat{V_1}+\dim\hat{V_2}-\dim(\hat{V_1}\cap\hat{V_2})=\dim V_1+\dim V_2-\dim(V_1\cap V_2)=
\dim(V_1+V_2)=\dim\phi(V_1+V_2)$. Thus~$\phi$ is join-preserving.
\\
(b) Let $E_2\neq0$. The only implication that remains to be proven  is that if $\phi_\cL$ is a $\{0,1\}$-lattice homomorphism then~$\phi$ is bijective.
Thus, let $\phi_\cL$ be a $\{0,1\}$-lattice homomorphism.
By assumption $\phi(E_1)=E_2$, and thus~$\phi$ is surjective and $\dim E_2\leq \dim E_1$.
Let $e_i=\dim E_i$. We have to show that $e_1=e_2$.
Assume by contradiction that $e_2<e_1$.
We proceed in several steps.
\\
i) Clearly~$\phi$ maps $1$-spaces of~$E_1$ to subspaces of~$E_2$ of dimension~$1$ or~$0$.
Suppose $\phi(V)=\phi(W)$ for some 1-spaces $V\neq W$.
Then the properties of a $\{0,1\}$-lattice homomorphism imply $\phi(V)=\phi(V)\cap\phi(W)=\phi(V\cap W)=\phi(0)=0$.
Since~$E_i$ has $(q^{e_i}-1)/(q-1)$ 1-spaces, we conclude that at most $(q^{e_2}-1)/(q-1)$ 1-spaces of~$E_1$ are mapped to $1$-spaces, and
thus~$\phi$ maps at least $(q^{e_1}-q^{e_2})/(q-1)$ 1-spaces to~$0$.
 \\
ii) Let now $V\in\cL(E_1)$ be a maximal-dimensional subspace such that $\phi(V)=0$.
Then $V\neq E_1$ since $E_2\neq0$.
Let $\dim V=v$.
The number of 1-spaces in~$V$ is $(q^v-1)/(q-1)$.
Using $e_2\leq e_1-1$ as well as $v\leq e_1-1$, we compute
\[
   \frac{q^{e_1}-q^{e_2}}{q-1}-\frac{q^v-1}{q-1}\geq \frac{q^{e_1}-q^{e_2}}{q-1}-\frac{q^{e_1-1}-1}{q-1}
   =q^{e_1-1}-\frac{q^{e_2}-1}{q-1}\geq1.
\]
Now~i) tells us that there exists at least one $1$-space $W\in\cL(E_1)$ such that $W\not\leq V$ and $\phi(W)=0$.
Now the join-preserving property implies $\phi(V+W)=\phi(V)+\phi(W)=0$, in contradiction to the maximality of~$V$.
All of this shows that $e_1=e_2$ and thus~$\phi$ is bijective.
\end{proof}

The $\cL$-isomorphisms from \cref{P-LattHom}(b) are, up to $\cL$-equivalence, semi-linear maps.
This is a consequence of the Fundamental Theorem of Projective Geometry;
see for instance \cite[Ch.~II.10]{Art57} or \cite[Thm.~1]{Put} (for the interesting original source of this theorem, which is attributed to von Staudt, one
may also consult \cite{Har08}).

\begin{theo}\label{T-FTPG}
Let $\dim E_1=\dim E_2=n$, where $n\geq3$. Let $\tau:\cL(E_1)\longrightarrow\cL(E_2)$ be a lattice isomorphism.
Then there exists a semi-linear map $\phi:E_1\longrightarrow E_2$ such that $\tau=\phi_\cL$.
\end{theo}

\begin{cor}\label{T-SemiLinIso}
Let $\dim E_1=\dim E_2=n$, where $n\geq3$ or $q=n=2$.
Let $\tau:E_1\longrightarrow E_2$ be an $\cL$-map and
suppose there exist vectors $v_1,\ldots,v_n$ of $E_1$ such that
$\tau(v_1),\ldots,\tau(v_n)$ are linearly independent.
Then~$\tau$ is bijective and there exists a semi-linear isomorphism $\phi:E_1\longrightarrow E_2$
such that $\tau\sim_\cL\phi$.
In particular, if $q=2$, then $\tau=\phi$ is a linear isomorphism.
\end{cor}

\begin{proof}
Since $\tau(E_1)$ is a subspace and contains $\tau(v_1),\ldots,\tau(v_n)$, we conclude that $\tau(E_1)=E_2$.
Hence~$\tau$ is a bijection.
For $n=2=q$ it is clear that~$\tau$ is a linear isomorphism
(any nonzero vector in $E_1$ or $E_2$ is the sum of the other two nonzero vectors).
For $n\geq3$ we may use \cref{T-FTPG}.
\end{proof}

We return to general $\cL$-maps and list some basic facts.

\begin{prop}\label{P-EquivMaps}
Let $\phi,\psi:E_1\longrightarrow E_2$ be $\cL$-maps such that $\phi\sim_\cL\psi$.
\begin{alphalist}
\item For all $v\in E_1$ there exist $\lambda_v\in\F^*$ such that $\phi(v)=\lambda_v\psi(v)$.
\item $\phi^{-1}(V)=\psi^{-1}(V)$ for all $V\in\cL(E_2)$.
\item If $\psi$ is an $\cL$-isomorphism, then so is $\phi$, and $\phi^{-1}\sim_\cL\psi^{-1}$.
\item If $\psi$ is injective (resp.~surjective), then so is~$\phi$.
\item If $\phi$ and $\psi$ are linear, then $\phi=\lambda\psi$ for some $\lambda\in\F^*$.
\end{alphalist}
\end{prop}

\begin{proof}
(a) It suffices to consider $v\neq0$.
Note that $\phi(v)=0\Longleftrightarrow\psi(v)=0$, and in this case we may choose $\lambda_v=1$.
In the case $\psi(v)\neq0\neq\phi(v)$ the result follows from~\eqref{e-phiV}.
\\
(b) For $v\in E_1$ let $\lambda_v\in\F^*$ be as in~(a).
For any $V\in\cL(E_2)$ we have
$\phi^{-1}(V)=\{v\in E_1\mid \phi(v)\in V\}=\{v\in E_1\mid \lambda_v\psi(v)\in V\}=\{v\in E_1\mid \psi(v)\in V\}=\psi^{-1}(V)$.
\\
(c) If~$\psi$ is an $\cL$-isomorphism, then the lattices $\cL(E_1)$ and $\cL(E_2)$ are isomorphic (see \cref{R-SIso}).
Next, $E_2=\psi(E_1)=\phi(E_1)$, and thus~$\phi$ is bijective as well.
Furthermore, $\psi_\cL$ and $\phi_\cL$ are lattice isomorphisms from~$\cL(E_1)$ to $\cL(E_2)$ satisfying
$(\psi^{-1})_\cL=(\psi_\cL)^{-1}=(\phi_\cL)^{-1}=(\phi^{-1})_\cL$.
\\
(d) The statement about surjectivity is clear since $\phi(E_1)=\psi(E_1)$. For injectivity use~(c) with $E_2=\psi(E_1)$.
\\
(e)
Let $W=\ker\psi$, hence also $W=\ker\phi$, and let $E_1=V\oplus W$.
Let $v_1,\ldots,v_r$ be a basis of~$V$.
By~(a) there exist $\lambda_i\in\F^*$ such that $\phi(v_i)=\lambda_i\psi(v_i)$ for $i\in[r]$.
Furthermore, $\phi(v_1+v_i)=\hat{\lambda}\psi(v_1+v_i)$ for some $\hat{\lambda}\in\F^*$.
Linearity implies $\psi(\lambda_1v_1+\lambda_iv_i)=\phi(v_1+v_i)=\hat{\lambda}\psi(v_1+v_i)=\psi(\hat{\lambda}v_1+\hat{\lambda}v_i)$, and injectivity of $\psi$ on~$V$ yields $\lambda_1=\lambda_i=\hat{\lambda}$ for all $i\in[r]$.
Hence $\phi|_V=\hat{\lambda}\psi|_V$.
Now we obtain $\phi(v+w)=\phi(v)+0=\hat{\lambda}\psi(v)+0=\hat{\lambda}\psi(v)+\hat{\lambda}\psi(w)=\hat{\lambda}\psi(v+w)$ for all
$v\in V$ and $w\in W$, and this proves the desired statement.
\end{proof}

The following results show how we may alter an $\cL$-map without changing the induced lattice homomorphism.
Part~(a) also shows that $\cL$-maps are $\cL$-equivalent if they agree on the 1-spaces.

\begin{prop}\label{P-S1dim}
Let $\psi:E_1\longrightarrow E_2$ be an $\cL$-map.
\begin{alphalist}
\item Let $\phi:E_1\longrightarrow E_2$ be a map such that $\phi(0)=0$ and $\phi(\subspace{v})=\psi(\subspace{v})$ for all
        $v\in E_1$. Then $\phi$ is an $\cL$-map and $\phi\sim_\cL\psi$.
\item Suppose~$\psi$ is an $\cL$-isomorphism. Fix $w\in E_1\setminus0$ and $\tau\in\F^*$ and set $\hat{w}=\psi^{-1}(\tau\psi(w))$.
Then $\subspace{\hat{w}}=\subspace{w}$.
Define the map $\phi:E_1\longrightarrow E_2$ via
\[
   \phi(v)=\psi(v)\ \text{ for } v\in E_1\setminus \subspace{w},\quad
   \phi(\mu w)=\psi(\mu \hat{w})\ \text{ for }\mu\in\F.
\]
Then $\phi$ is an $\cL$-isomorphism and $\phi\sim_\cL\psi$.
If $\tau=1$, then $\phi=\psi$.
\end{alphalist}
\end{prop}

\begin{proof}
(a) is immediate from $\phi(V)=\bigcup_{v\in V}\phi(\subspace{v})=\bigcup_{v\in V}\psi(\subspace{v})=\psi(V)$ for all $V\in\cL(E_1)$.
\\
(b) Bijectivity of~$\psi$ implies $\hat{w}\neq0$ and $\psi(\hat{w})=\tau\psi(w)$, thus
$\psi(\subspace{w})=\subspace{\psi(w)}=\subspace{\psi(\hat{w})}=\psi(\subspace{\hat{w}})$.
This in turn implies $\subspace{w}=\subspace{\hat{w}}$.
The map $\phi$ is clearly bijective and satisfies $\phi(0)=0$.
We show that~$\phi$ satisfies the condition of (a).
First let $v\not\in\subspace{w}$. Then also $\lambda v\not\in \subspace{w}$ for all $\lambda\in\F^*$ and
$\phi(\subspace{v})=\{\phi(\lambda v)\mid\lambda\in\F\}=\{\psi(\lambda v)\mid \lambda\in\F\}=\psi(\subspace{v})$.
Next, for $v=\mu w$ with $\mu\neq0$ we have
$\phi(\subspace{v})=\{\phi(\lambda\mu w)\mid \lambda\in\F\}=\{\psi(\lambda\mu\hat{w})\mid \lambda\in\F\}
  =\psi(\subspace{\hat{w}})=\psi(\subspace{w})=\psi(\subspace{v})$.
Thus we may apply~(a) and the statement follows.
\end{proof}

We now turn to $\cL$-maps between \qM{}s. There are different options for such a map to respect the \qM{} structure.
For~(a) and~(b) below we adopt the terminology known for classical matroids; see, e.g.\  \cite[Def.~8.1.1 and Def.~9.1.1]{Wh86}.
Our notion of rank-preserving maps in~(c), however, is different from rank-preserving weak maps for classical matroids:
the latter are weak maps that preserve the rank of the matroid; see \cite[p.~260]{Wh86}.
The definition below will be convenient for us.

\begin{defi}\label{D-StrMap}
Let $\cM_i=(E_i,\rho_i)$ be \qM{}s with flats $\cF_i:=\cF(\cM_i)$.
Let $\phi:E_1\longrightarrow E_2$ be an $\cL$-map. We define the following \emph{types}.
\begin{alphalist}
\item $\phi$ is a \emph{strong map} from $\cM_1$ to $\cM_2$ if $\phi^{-1}(F)\in\cF(\cM_1)$
          for all $F\in\cF(\cM_2)$ (this implies in particular that $\phi^{-1}(F)$ is a subspace of~$E_1$).
\item  $\phi$ is a \emph{weak map} from $\cM_1$ to $\cM_2$ if $\rho_2(\phi(V))\leq\rho_1(V)$ for all $V\in\cL(E_1)$.
\item $\phi$ is \emph{rank-preserving} from $\cM_1$ to $\cM_2$ if $\rho_2(\phi(V))=\rho_1(V)$ for all $V\in\cL(E_1)$.
\end{alphalist}
For any $\cL$-map $\phi:E_1\longrightarrow E_2$ we will also use the notation $\phi:\cM_1\longrightarrow\cM_2$.
This allows us to discuss its type.
\end{defi}

Note that each of the types above are actually properties of the induced map $\phi_\cL$.
This raises the question as to whether one should define maps between \qM{}s $(E_1,\rho_1)$ and $(E_2,\rho_2)$ as maps
(with certain properties) between the underlying subspace lattices $\cL(E_1)$ and $\cL(E_2)$.
However lattice homomorphisms are too restrictive as they exclude some non-injective maps  (see also \cref{R-LattHomo}(b)), while
simply order-preserving maps appear to be too general.
In \cref{S-Lclass} we will briefly consider the setting where the maps are those induced by $\cL$-maps.
Note that the distinction of maps between the ground spaces versus maps between the subspace lattices does not occur for classical
matroids because a map on a set~$S$ is uniquely determined by its induced map on the subset lattice of~$S$.

We return to \cref{D-StrMap}.
Clearly, the composition of maps of the same type is again a map of that type.
Furthermore, $\cL$-equivalent $\cL$-maps are of the same type (see \cref{P-EquivMaps}(b) for strong maps).
Note, however, that if~$\phi:\cM_1\longrightarrow\cM_2$ is a bijective strong (resp.\ weak) map, then $\phi^{-1}:\cM_2\longrightarrow\cM_1$
may not be strong (resp.\ weak):
take for instance the identity map $\cU_{k}(\F^n)\longrightarrow\cU_{k-1}(\F^n)$.
Being rank-preserving and being strong are not related:
there exist strong maps that are not rank-preserving (e.g., the identity map from any nontrivial \qM{} $\cM=(E,\rho)$ to the trivial \qM{} on~$E$)
and there exist rank-preserving maps that are not strong (e.g., $\phi:\F_2^2\longrightarrow\F_2^2,\ (x,y)\longmapsto (x,0)$ and where
$\cM_1$ and $\cM_2$ are the \qM{}s on $\F_2^2$ of rank~$1$ with $\subspace{e_2}$ and $\subspace{e_1+e_2}$ as the unique flat of rank~$0$, respectively).
Unsurprisingly, weak maps are in general not strong: take for instance the identity on $\F_2^4$, which induces a weak,
but not strong map from $\cU_{2}(\F_2^4)$ to the \qM{} $\cM$ from \cref{E-NonRepr}.
However, it can be shown that -- just like in the classical case \cite[Lemma~8.1.7]{Wh86} -- strong maps are weak.
This and other properties and characterizations of strong maps can be found in~\cite{Ja22}. We do not need those facts in this paper.
The following simple result will suffice for our considerations.

\begin{prop}\label{P-StroWeak}
Let $\cM_i=(E_i,\rho_i)$ be \qM{}s and $\phi:E_1\longrightarrow E_2$ be an $\cL$-isomorphism.
Consider~$\phi$ as a map from $\cM_1$ to $\cM_2$.
Then
 \[
             \phi \text{ and }\phi^{-1}\text{ are weak maps}\Longleftrightarrow \phi\text{ is rank-preserving }\Longleftrightarrow
             \phi\text{ and }\phi^{-1}\text{ are strong maps.}
\]
\end{prop}

\begin{proof}
Recall from \cref{R-SIso} that $\phi^{-1}$ is also an $\cL$-map.
The first equivalence is clear. Consider the second equivalence.
\\
 ``$\Rightarrow$'' Let $\phi$ be a rank-preserving isomorphism. Then $\rho_2(\phi(V))=\rho_1(V)$ and $\dim\phi(V)=\dim V$ for all $V\in\cL(E_1)$.
Using \cref{T-AxFlats} we conclude that~$V$ is a flat in~$\cM_1$ iff~$\phi(V)$ is a flat in~$\cM_2$.
\\
``$\Leftarrow$'' Let now $\phi$ and $\phi^{-1}$ be strong maps.
Then
$\phi(\cF(\cM_1))=\cF(\cM_2)$, that is, $\phi$ induces an isomorphism between the lattices of flats
of~$\cM_1$ and~$\cM_2$.
Using the height function on these lattices (see \cref{T-FlatsMatroid}(c) and \cref{T-FlatsRank}(a))
we conclude that $\rho_1(F)=\rho_2(\phi(F))$ for all $F\in\cF(\cM_1)$ and thus $\rho_1(V)=\rho_2(\phi(V))$ for all $V\in\cL(E_1)$.
Hence~$\phi$ is rank-preserving.
\end{proof}

Just like there exist linear maps between any vector spaces (over the same field), there exist weak and strong maps between any \qM{}s~$\cM_1$ and~$\cM_2$.

\begin{exa}\label{E-WSMaps}
Let $\cM_i=(E_i,\rho_i),\,i=1,2,$ be \qM{}s.
\begin{alphalist}
\item The zero map $\phi_0$ from~$E_1$ to~$E_2$ is a weak map from~$\cM_1$ to~$\cM_2$. \
        It is also a strong map since $\phi_0^{-1}(W)=E_1\in\cF(\cM_1)$ for all $W\leq E_2$.
\item Let $X=\{x\in E_1\mid \rho_1(\subspace{x})=0\}$, that is,~$X$ is the closure of $\{0\}$ in~$\cM_1$ in the sense of \cref{T-AxFlats}.
        In particular, $X\leq E_1$.
        Let $z\in E_2\setminus\{0\}$.
        As in \cref{E-qMMap}(a) choose vectors $v_1,\ldots,v_\ell\in E_1\setminus X$ such that $\subspace{v_1},\ldots,\subspace{v_\ell}$ are the distinct lines in $E$ that are not in~$X$ and
        consider the map~$\phi$ as in that example.
        Then~$\phi$ is a weak map. Indeed, for any $V\leq X$ we have $\rho_1(V)=0=\rho_2(\{0\})=\rho_2(\phi(V))$, while for $V\in\cL(E_1)\setminus\cL(X)$
        we have $\rho_1(V)\geq 1\geq\rho_2(\subspace{z})=\rho_2(\phi(V))$.
        Moreover, the pre-images in \cref{E-qMMap}(a) show that~$\phi$ is strong.
        Note that if $X=E_1$, i.e.,~$\cM_1$ is the trivial \qM{}, then~$\phi$ is the zero map.
\end{alphalist}
\end{exa}

We now turn to minors of \qM{}s and determine the type of the corresponding maps.
For restrictions and contractions of \qM{}s, discussed next, see for instance \cite[Def.~5.1 and Thm.~5.2]{GLJ22Gen}.

\begin{prop}\label{P-DelContr}
Let $\cM=(E,\rho)$ be a $q$-matroid and let $X\leq E$.
\begin{alphalist}
\item Let $\cM|_X$ be the restriction of~$\cM$ to~$X$.
         Then the embedding $\iota:X\longrightarrow E,\ x\longmapsto x$, is a linear strong and rank-preserving (hence weak)
         map from $\cM|_X$ to~$\cM$.
\item Let $\cM/X$ be the contraction of~$X$ from~$\cM$. Then the projection $\pi:E\longrightarrow E/X,\ x\longmapsto x+X$ is a
        linear strong and weak map from~$\cM$ to $\cM/X$.
\end{alphalist}
\end{prop}

\begin{proof}
(a)
Recall that $\cM|_X=(X,\hat{\rho})$, where $\hat{\rho}(V)=\rho(V)$ for all $V\leq X$.
This shows that~$\iota$ is rank-preserving.
Let $F\in\cF(\cM)$.
Then $\iota^{-1}(F)=F\cap X$ and $\overline{F\cap X}^{\cM}\subseteq\overline{F}^{\cM}=F$, where $\overline{A}^{\cM}$
denotes the closure of the space~$A$ in the $q$-matroid~$\cM$ (for the closure see \cref{T-AxFlats}).
Thus for any $v\in X$ the identity $\rho((F\cap X)+\subspace{v})=\rho(F\cap X)$ implies $v\in F$.
Hence  $\overline{F\cap X}^{\cM|_X}=F\cap X$ and thus $\iota^{-1}(F)\in\cF(\cM|_X)$.
\\
(b) Recall that $\cM/X=(E/X,\tilde{\rho})$, where $\tilde{\rho}(\pi(V))=\rho(V+X)-\rho(X)$.
Thus by submodularity $\tilde{\rho}(\pi(V))\leq\rho(V)-\rho(V\cap X)$, showing that~$\pi$ is weak.
In order to show that~$\pi$ is strong, let $F\in\cF(\cM/X)$. Thus
$\tilde{\rho}(F+\subspace{v+X})>\tilde{\rho}(F)$ for all $v+X\in (E/X)\setminus F$.
Since $\pi^{-1}(F+\subspace{v+X})=\pi^{-1}(F)+\subspace{v}$, this implies $\rho(\pi^{-1}(F)+\subspace{v})>\rho(\pi^{-1}(F))$ for all
$v\in E\setminus\pi^{-1}(F)$. Thus $\pi^{-1}(F)\in\cF(\cM)$.
\end{proof}

Restricting an $\cL$-map to its image does not change its type.

\begin{prop}\label{P-ImageStr}
Let $\cM_i=(E_i,\rho_i)$ be $q$-matroids and $\phi:\cM_1\longrightarrow \cM_2$ be a strong (resp.\ weak or rank-preserving)
map.  Let $X:=\im\phi$.
Then $X$ is a subspace of~$E_2$ and we call the restriction $\cM_2|_X$ the \emph{image of~$\phi$}.
The map $\hat{\phi}: E_1\longrightarrow X,\ v\longmapsto \phi(v)$, is a strong (resp.\ weak or rank-preserving)
map from~$\cM_1$ to $\cM_2|_X$.
In other words, $\phi$ restricts to a map $\cM_1\longrightarrow\cM_2|_X$ of the same type.
\end{prop}

\begin{proof}
The statement is clear for weak and rank-preserving maps.
Let now~$\phi$ be strong and $F\in\cF(\cM_2|_X)$.
Then $\rho_2(F+\subspace{v})>\rho_2(F)$ for all $v\in X\setminus F$.
Hence the closure $\overline{F}^{\cM_2}$ satisfies $\overline{F}^{\cM_2}\setminus F\subseteq E_2\setminus X$.
Using that $\im\phi =X$, we obtain $\phi^{-1}(\overline{F}^{\cM_2})=\hat{\phi}^{-1}(F)$.
Since the former is a flat in~$\cM_1$, we conclude that~$\hat{\phi}$ is a strong map.
\end{proof}

Finally we record the simple observation that representability is not preserved under strong or weak bijective maps.
Take for instance the identity map on~$\F_2^4$.
It induces a bijective strong and weak map from the representable \qM{}
$\cU_{4}(\F_2^4)$ (see  \cref{E-Uniform}) to the non-representable \qM{} in \cref{E-NonRepr}.

\section{Non-Existence of Coproducts in Categories of $q$-Matroids}\label{S-NoCopr}

In this section we consider categories of \qM{}s with various types of morphisms.
We will show that --  with one exception -- none of these categories has a coproduct.

\begin{defi}\label{D-qMatCat}
We denote by \qMats, \qMatrp, \qMatw, \qMatls, \qMatlrp{}, and \qMatlw{} the categories with \qM{}s as objects and
where the morphisms are the strong, rank-preserving, weak,  linear strong, linear rank-preserving and linear weak maps, respectively.
\end{defi}

In this section we show that none of the first  5 categories has a coproduct, while in the next section we establish the existence of a
coproduct in \qMatlw. It is in fact the direct sum as introduced recently in \cite{CeJu21}.
The non-existence of a coproduct in \qMats{} and \qMatls{} stands in contrast to the case of classical matroids, where the direct sum
(see \cite[Sec.~4.2]{Ox11}) forms a
coproduct in the category with strong maps as morphisms, see \cite[Ex.~8.6, p.~244]{Wh86} (which goes back to \cite{CrRo70}).
We know from \cref{P-StroWeak} that isomorphisms in the first three categories coincide, and so do those in the second three categories.
This gives rise to the following notions of isomorphic \qM{}s.

\begin{defi}\label{D-IsomMatr}
We call \qM{}s $\cM_1$ and $\cM_2$ \emph{isomorphic} if they are isomorphic in the category \qMatrp, that is,
there exists a rank-preserving $\cL$-isomorphism $\phi:\cM_1\longrightarrow\cM_2$ (equivalently, $\phi$ and $\phi^{-1}$ are strong).
$\cM_1$ and~$\cM_2$ are \emph{linearly isomorphic}, denoted by $\cM_1\cong\cM_2$, if they are isomorphic in the category \qMatlrp.
\end{defi}

Due to \cref{T-FTPG} the above notion of isomorphism coincides with lattice-equivalence in \cite[Def.~5]{CeJu21}.
The same theorem tells us that any rank-preserving $\cL$-isomorphism is induced by a semi-linear map.

\begin{rem}\label{R-ConstrIsom}
Let $\phi:E_1\longrightarrow E_2$ be an $\cL$-isomorphism and $\cM_1=(E_1,\rho_1)$ be a \qM{}.
Define $\cM_2=(E_2,\rho_2)$ via $\rho_2(V):=\rho_1(\phi^{-1}(V))$ for all $V\in\cL(E_2)$.
Then $\cM_2$ is a \qM{} and isomorphic to~$\cM_1$.
The flats of $\cM_2$ are given by $\cF(\cM_2)=\{\phi(F)\mid F\in\cF(\cM_1)\}$.
\end{rem}

The following linear maps will be used throughout this paper. For $E=E_1\oplus E_2$ let
\begin{equation}\label{e-iota12}
   \iota_i:E_i\longrightarrow E,\ x \longrightarrow x
\end{equation}
be the natural embeddings. If $E_i=\F^{n_i}$ and $E=\F^{n_1+n_2}$ we define the maps as
\begin{equation}\label{e-iota12F}
   \iota_1:\F^{n_1}:\longrightarrow \F^{n_1+n_2},\ x \longrightarrow (x,0),\qquad
   \iota_2:\F^{n_2}:\longrightarrow \F^{n_1+n_2},\ y \longrightarrow (0,y).
\end{equation}

Next we present a simple construction that will be crucial later on.
It shows that representable $q$-matroids~$\cM_1$ and $\cM_2$ can be embedded in a $q$-matroid~$\cM$ in such a way that the ground spaces of~$\cM_i$ form a direct sum of the ground spaces of~$\cM$.
However, as we will illustrate by an example below, the resulting \qM{}~$\cM$ is not uniquely determined by the \qM{}s~$\cM_1$
and~$\cM_2$, but rather depends on the representing matrices (which has also been observed in \cite[Sec.~3.2]{CeJu21}).
It is exactly this non-uniqueness that allows us to prove the non-existence of coproducts.

\begin{prop}\label{P-DirProdMat}
Let $\F_{q^m}$ be a field extension of~$\F=\F_q$ and $G_i\in\F_{q^m}^{a_i\times n_i},\,i=1,2,$ be matrices of full row rank.
Set
\[
    G=\begin{pmatrix}G_1&0\\0&G_2\end{pmatrix}\in\F_{q^m}^{(a_1+a_2)\times(n_1+n_2)}.
\]
Denote by $\cM_i=(\F^{n_i},\rho_i),\,i=1,2,$ and $\cN=(\F^{n_1+n_2},\rho)$ the $q$-matroids represented by $G_1,G_2$, and~$G$, thus
$\rho_i(\rowsp(Y))=\rk(G_iY\T)$ for $Y\in\F^{y\times n_i}$ and $\rho(\rowsp(Y))=\rk(GY\T)$ for $Y\in\F^{y\times(n_1+n_2)}$.
Then $\iota_i:\cM_i\longrightarrow\cN,\,i=1,2$, is a linear, rank-preserving, and
strong map with image $\cN|_{\F^{n_i}}$.
Thus $\cN|_{\F^{n_i}}$ is linearly isomorphic to~$\cM_i$ for $i=1,2$.
\end{prop}

\begin{proof}
Let $Y\in\F^{y\times n_1}$ be a matrix of rank~$y$. Then $\rowsp(Y)\leq\F^{n_1}$ and $\iota_1(\rowsp(Y))=\rowsp(Y\,|\,0)$. Moreover,
\[
   \rho(\rowsp(Y\mid 0))=\rk\bigg(\!\begin{pmatrix}G_1&0\\0&G_2\end{pmatrix}\begin{pmatrix}Y\T\\0\end{pmatrix}\!\bigg)
    =\rk(G_1Y\T)=\rho_1(\rowsp(Y)).
\]
Thus~$\iota_1$ is rank-preserving.
Similarly,~$\iota_2$ is rank-preserving.
In order to show that~$\iota_i:\cM_i\longrightarrow\cN$ is a  strong map, let us first consider the pre-images of a subspace $\rowsp(Y)$,
where $Y\in\F^{y\times(n_1+n_2)}$.
They can be computed as follows.
There exist $U_i\in\GL_y(\F)$ such that
\[
   U_1Y=\begin{pmatrix}A_1&0\\ A_3&A_4\end{pmatrix},\quad
   U_2Y=\begin{pmatrix}0&B_2\\ B_3&B_4\end{pmatrix},
\]
where the first block column consists of~$n_1$ columns and the second one of~$n_2$ columns, and where~$A_4$ and~$B_3$ have full row rank.
Then $\iota_1^{-1}(\rowsp(Y))=\rowsp(A_1)$ and $\iota_2^{-1}(\rowsp(Y))=\rowsp(B_2)$.
Suppose now that $\rowsp(Y)\in\cF(\cM)$.
By symmetry it suffices to show that $\rowsp(A_1)\in\cF(\cM_1)$.
To this end let $v_1\in\F^{n_1}$ such that $ \rho_1(\rowsp(A_1)+\subspace{v_1})=\rho_1(\rowsp(A_1))$.
Setting $v=\iota_1(v_1)=(v_1,0)$, we have
\[
   \rho_1(\rowsp(A_1)+\subspace{v_1})=\rk\big(G_1(A_1\!\T\ v_1\!\T)\big)\ \text{ and }\
   \rho(\rowsp(Y)+\subspace{v})=\rk\begin{pmatrix}G_1A_1\!\T&G_1v_1\!\T&G_1A_3\!\T\\0&0&G_2A_4\!\T\end{pmatrix}.
\]
We conclude that $\rho(\rowsp(Y)+\subspace{v})=\rho(\rowsp(Y))$ and hence $v\in\rowsp(Y)$ since the latter is a flat.
But then $v_1\in\rowsp(A_1)$, and this shows that $\rowsp(A_1)$ is a flat.
This shows that $\iota_i$ are linear, injective, strong, and rank-preserving maps.
Clearly, $\cN|_{\F^{n_i}}$ is the image of~$\iota_i$, and thus~$\iota_i$ induces an isomorphism between
$\cM_i$ and $\cN|_{\F^{n_i}}$ in \qMatlrp{}.
\end{proof}

The next proposition shows for a special case that the \qM{}~$\cN$ of the last proposition depends on the representing matrices.
This stands in contrast to the classical case, where the block diagonal matrix of every choice of representing matrices is a representing matrix of the direct sum; see \cite[p.~126, Ex.~7]{Ox11}.

\begin{prop}\label{P-BlockDiag}
Let $\F_{q^m}$ be a field extension of $\F=\F_q$ with primitive element~$\omega$ and
let $\Omega=\{1,\ldots,q^m-2\}\setminus\{k(q^m-1)/(q-1)\mid k\in\N\}$ (hence $\omega^i\not\in\F$ for $i\in\Omega$).
For $i\in\Omega$ define
\[
   G^{(i)}=\begin{pmatrix}1&\omega&0&0\\ 0&0&1&\omega^i\end{pmatrix},
\]
and let $\cN^{(i)}=(\F^4,\rho^{(i)})$ be the \qM{} represented by~$G^{(i)}$.
Define the subspaces $T_1=\subspace{1000,\,0100}$ and $T_2=\subspace{0010,\,0001}$.
Then for all $i\in\Omega$
\begin{alphalist}
\item $\rho^{(i)}(T_1)=\rho^{(i)}(T_2)=1$ and $\rho^{(i)}(\F^4)=2$.
\item $\rho^{(i)}(V)=1$ for all $1$-spaces~$V$ and $\rho^{(i)}(V)=2$ for all $3$-spaces~$V$.
\item Let $\cL_2=\{V\in\cL(\F^4)\mid \dim V=2,\,T_1\neq V\neq T_2\}$. Then
         \[
             \rho^{(i)}(V)=2\text{ for all  }V\in\cL_2
              \Longleftrightarrow 1,\omega,\omega^i,\omega^{i+1}\text{ are linearly independent over } \F.
         \]
\item The flats of~$\cN^{(i)}$ are given by $\cF(\cN^{(i)})=\{0,\F^4\}\cup\cF_1^{(i)}\cup\cF_2^{(i)}$, where
\begin{equation}\label{e-FPrime}
\left.\begin{split}
   \cF_2^{(i)}&=\{V\in\cL(\F^4)\mid \dim V=2,\,\rho^{(i)}(V)=1\},\\
   \cF_1^{(i)}&=\{V\in\cL(\F^4)\mid \dim V=1, V\not\leq W\text{ for all }W\in \cF_2^{(i)}\}.
\end{split}\qquad\right\}
\end{equation}
Thus for $m>3$ there exist at least two non-isomorphic \qM{}s of the form $\cN^{(i)}$.
\end{alphalist}
\end{prop}

\begin{proof}
(a) Clearly $\rho^{(i)}(T_1)=\rho^{(i)}(T_2)=1$ and $\rho^{(i)}(\F^4)=\rk(G)=2$.
\\
(b) By assumption on~$i$ the elements $1,\omega^i$ are linearly independent over~$\F$.
Thus, $G^{(i)}x\neq0$ for any nonzero vector $x\in\F^4$ and hence $\rho^{(i)}(V)=1$ for all 1-spaces~$V$.
Suppose there exists a 3-space $V$ such that $\rho^{(i)}(V)=1$.
Clearly,~$V$ does not contain both $T_1$ and~$T_2$. Let $T_1\not\leq V$.
Then $\dim(V\cap T_1)=1$, and submodularity of~$\rho^{(i)}$ implies $2=\rho^{(i)}(\F^4)=\rho^{(i)}(V+T_1)\leq 1+1-\rho^{(i)}(V\cap T_1)=1$,
which is a contradiction.
\\
(c) Consider now an arbitrary 2-space $V=\subspace{(a_0,a_1,a_2,a_3),(b_0,b_1,b_2,b_3)}$.
Since $\rho^{(i)}(V)$ is the rank of the matrix
\begin{equation}\label{e-MatProd}
    \begin{pmatrix}1&\omega&0&0\\ 0&0&1&\omega^i\end{pmatrix}
    \begin{pmatrix}a_0&b_0\\a_1&b_1\\a_2&b_2\\a_3&b_3\end{pmatrix}
    =\begin{pmatrix}a_0+a_1\omega&b_0+b_1\omega\\ a_2+a_3\omega^i&b_2+b_3\omega^i\end{pmatrix},
\end{equation}
we conclude that $\rho^{(i)}(V)=1$ if and only if its determinant is zero, thus
\begin{equation}\label{e-det}
  \rho^{(i)}(V)=1\Longleftrightarrow(a_0b_2-a_2b_0)+(a_1b_2-a_2b_1)\omega+(a_0b_3-a_3b_0)\omega^i+(a_1b_3-a_3b_1)\omega^{i+1}=0.
\end{equation}
Now we can prove the stated equivalence.
``$\Leftarrow$'' Suppose $1,\omega,\omega^i,\omega^{i+1}$ are linearly independent.
Then $\rho^{(i)}(V)=1$ if and only if all coefficients in \eqref{e-det} are zero.
We consider the following cases.
(i) If $a_0\neq0=b_0$, then $b_2=b_3=0$.
Thus $b_1\neq0$ since $\dim V=2$, and subsequently $a_2=a_3=0$. But then $V=T_1$.
(ii) Suppose $a_0=b_0=0$. If $a_1=0=b_1$, then $V=T_2$. Thus assume without loss of generality that $a_1\neq0$. But then
$b_2=(b_1/a_1)a_2$ and $b_3=(b_1/a_1)a_3$ and thus $\dim V=1$, which is a contradiction.
(iii) If $a_0\neq0\neq b_0$, then without loss of generality $a_0=b_0$ and as in (ii) we obtain $\dim V=1$ or $V=T_1$.
\\
``$\Rightarrow$'' Suppose $1,\omega,\omega^i,\omega^{i+1}$ are linearly dependent, say
$f_0+f_1\omega+f_2\omega^i+f_3\omega^{i+1}=0$ with $f_0,\ldots,f_3\in\F$, not all zero.
Using \eqref{e-MatProd} one obtains $\rho(V)=1$ for $V=\subspace{(1,0,-f_1,-f_3),\,(0,1,f_0,f_2)}\in\cL_2$.
\\
(d) The statement about the flats is clear from the rank values, and the statement about non-isomorphic $q$-matroids~$\cN^{(i)}$
follows from~(c) because  $1,\omega,\omega^i,\omega^{i+1}$ are linearly dependent for $i=1$ and linearly
independent for $i=2$ if $m>3$.
\end{proof}

For later use we record the following fact.

\begin{lemma}\label{L-Fprime}
Let $m\geq4$ and the data be as in \cref{P-BlockDiag}.
Define $\cF'=\cup_{i\in\Omega}\cF(\cN^{(i)})$. Then
\[
  \cF'=\{0,\F^4,T_1,T_2\}\cup\{V\in\cL(\F^4)\mid \dim V\leq 2, V\cap T_1=0=V\cap T_2\}.
\]
\end{lemma}

\begin{proof}
Recall the sets $\cF_1^{(i)},\cF_2^{(i)}$ from \eqref{e-FPrime}.
Note that $T_1,T_2\in\cF_2^{(i)}$ for all $i\in\Omega$.
Denote the set on the right hand side of the stated identity by $\cF''$.
``$\subseteq$'' Let $V\in\cF'$.
If $\dim V=1$, then $V\in\cF_1^{(i)}$ for some~$i$ and thus $V\not\leq T_\ell$ for $\ell=1,2$. Hence $V\in\cF''$.
Let now $\dim V=2$ and thus $\rho^{(i)}(V)=1$ for some~$i$.
The statement is clear for $V\in\{T_1,T_2\}$, and thus let $V\not\in\{T_1,T_2\}$.
Suppose $V\cap T_1=\subspace{v}$ for some $v\neq0$.
Then $\dim(V+T_1)=3$ and thus $2=\rho^{(i)}(V+T_1)\leq\rho^{(i)}(V)+\rho^{(i)}(T_1)-\rho^{(i)}(\subspace{v})=1$, a contradiction.
Thus, $V\in\cF''$.
\\
`$\supseteq$''
It is clear that the spaces $0,\F^4,T_1,T_2$ are in~$\cF'$. We consider the 1-spaces and 2-spaces separately.
\\
i) Let $V\in\cL(\F^4)$ with $\dim V=2$ and $V\cap T_1=V\cap T_2=0$.
Choosing the matrix $M\in\F^{2\times 4}$ in reduced row echelon form such that $V=\rowsp(M)$ we conclude that
\begin{equation}\label{e-MMat}
    M=\begin{pmatrix}1&0&a&b\\0&1&c&d\end{pmatrix},\ \text{ where }\det\begin{pmatrix}a&b\\c&d\end{pmatrix}\neq0.
\end{equation}
Now \eqref{e-det} reads as
\[
  \rho^{(i)}(V)=1\Longleftrightarrow c-a\omega+d\omega^i-b\omega^{i+1}=0
  \Longleftrightarrow \omega^i=\frac{a\omega-c}{d-b\omega}.
\]
Note that the denominator is indeed nonzero thanks to the determinant condition in \eqref{e-MMat}.
Since the fraction is in $\F_{q^m}$ and $\omega$ is a primitive element, $(a\omega-c)/(d-b\omega)=\omega^i$ for some
$i\in\{0,\ldots,q^m-2\}$.
Furthermore, the determinant condition in  \eqref{e-MMat} implies that $(a\omega-c)/(d-b\omega)\not\in\F$ and thus $i\in\Omega$.
All of this shows that $\rho^{(i)}(V)=1$ for some $i\in\Omega$ and thus $V\in\cF_2^{(i)}\subseteq\cF'$.
\\
ii) Let $\dim V=1$ and $V\not\leq T_\ell$ for $\ell=1,2$.
Let $W$ be a 2-space containing~$V$.
Then~$W$ is distinct from~$T_1$ and~$T_2$, and \cref{P-BlockDiag}(c) for $i=2$ implies $\rho^{(2)}(W)=2$.
Hence $V\in\cF_1^{(2)}\subseteq\cF'$, as desired.
\end{proof}

We now turn to coproducts in the above defined categories of \qM{}s.
Let us recall the definition.

\begin{defi}\label{D-CoPr}
Let $\CC$ be a category and $M_1,\,M_2$ be objects in~$\CC$.
A coproduct of~$M_1$ and~$M_2$ in~$\CC$ is a triple $(M,\xi_1,\xi_2)$, where~$M$ is an object in~$\CC$ and
$\xi_i:M_i\longrightarrow M$ are morphisms, such that
for all objects~$N$ and all morphisms $\tau_i:M_i\longrightarrow N$ there exists a unique morphism
$\epsilon: M\longrightarrow N$ such that $\epsilon\circ\xi_i=\tau_i$ for $i=1,2$.
\end{defi}

It is well-known (and straightforward to verify) that if $(M,\xi_1,\xi_2)$ is a coproduct of~$M_1$ and $M_2$ and
$\phi:M\longrightarrow \hat{M}$
is an isomorphism (i.e., a bijective morphism whose inverse is also a morphism), then $(\hat{M},\phi\circ\xi_1,\phi\circ\xi_2)$ is also a coproduct  of~$M_1$ and $M_2$.
Furthermore, every coproduct of~$M_1$ and $M_2$ is of this form.

The rest of this section is devoted to the non-existence of coproducts in the first 5 categories of \cref{D-qMatCat}.
Our first result narrows down the ground space of a putative coproduct in an expected way.
Furthermore, it shows that the accompanying morphisms are injective and can be chosen as linear maps.
We introduce the following notation.
Let $\cT=\{\textsf{s, rp, w, l-s, l-rp, l-w}\}$ be the set of types of morphisms.
For any $\Delta\in\cT$  we denote by \qMatd{} the corresponding category, and the morphisms in this category are called type-$\Delta$ maps.
Note that the maps in \cref{P-DirProdMat} and \cref{P-DelContr}(a) are type-$\Delta$ for each~$\Delta\in\cT$.

\begin{theo}\label{T-NoCopr2}
Let $\Delta\in\cT$.
Let $\cM_i,\,i=1,2,$ be representable \qM{}s with ground spaces~$E_i$.
Suppose $\cM_1$ and $\cM_2$ have a coproduct $(\cM,\xi_1,\xi_2)$ in \qMatd.
Then they have a coproduct of the form $(\tilde{\cM},\iota_1,\iota_2)$ where $\tilde{\cM}$ has ground space $E_1\oplus E_2$ and
$\iota_i:E_i\longrightarrow E_1\oplus E_2$ are the natural embeddings as in~\eqref{e-iota12} (hence linear).
\end{theo}

\begin{proof}
Suppose without loss of generality that $\cM_i$ is a \qM{} on the ground space $\F^{n_i}$.
Let $(\cM,\xi_1,\xi_2)$ be a coproduct of~$\cM_1,\cM_2$.
We may also assume that the ground space of~$\cM$ is $\F^n$ for some $n$.
Recall that $\xi_i$ may not be semi-linear maps, but the images $\xi_i(\F^{n_i})$ are subspaces of~$\F^n$.
We proceed in several steps.

\underline{Claim~1:} $\xi_i$ is injective for $i=1,2$ and $\xi_1(\F^{n_1})\cap\xi_2(\F^{n_2})=\{0\}$.
\\
Since each~$\cM_i$ is representable, we may apply \cref{P-DirProdMat} and obtain the existence of  a $q$-matroid~$\cN$ on $\F^{n_1+n_2}$ and linear maps
$\alpha_i:\cM_i\longrightarrow\cN$ where $\alpha_1(v_1)=(v_1,0)$ and $\alpha_2(v_2)=(0,v_2)$ for all $v_i\in\F^{n_i}$.
Thanks to \cref{P-DirProdMat} the maps~$\alpha_i$ are type-$\Delta$.
Hence the universal property of the coproduct implies the existence of a type-$\Delta$ map $\epsilon:\cM\longrightarrow\cN$ such that $\epsilon\circ\xi_i=\alpha_i,\,i=1,2$.
Now injectivity of~$\xi_i$ follows from injectivity of~$\alpha_i$.
Suppose $\xi_1(v_1)=\xi_2(v_2)$ for some $v_i\in\F^{n_i}$.
Then $\alpha_1(v_1)=\alpha_2(v_2)$, which means $(v_1,0)=(0,v_2)$.
Thus $v_1=0$ and $v_2=0$. This implies $\xi_1(v_1)=\xi_2(v_2)=0$, and the claim is proved.

\underline{Claim 2:}  $\F^n=\xi_1(\F^{n_1})\oplus\xi_2(\F^{n_2})$ and thus $n=n_1+n_2$.\\
Set $X_i=\xi_i(\F^{n_i})$, which is a subspace of~$\F^n$, and $X=X_1\oplus X_2$.
Consider the restriction $\cM|_X$.
The maps $\hat{\xi}_i:\cM_i\longrightarrow\cM|_X,\,v\longrightarrow\xi_i(v),$ are clearly type-$\Delta$ as well; see \cref{P-ImageStr}.
We want to show that $(\cM|_X,\hat{\xi}_1,\hat{\xi}_2)$ is a coproduct as well.
To do so, consider first  the diagram

\[
\begin{array}{l}
   \begin{xy}
   (0,40)*+{\cM_1}="a";
   (0,20)*+{\cM}="b";%
   (0,0)*+{\cM_2}="c";%
   (20,20)*+{\cM|_X}="d";%
   {\ar "a";"b"}?*!/_-2mm/{\xi_1};
    {\ar "a";"d"}?*!/_3mm/{\hat{\xi}_1};
    {\ar "c";"b"}?*!/_2mm/{\xi_2};
   {\ar "c";"d"}?*!/_-3mm/{\hat{\xi}_2};%
   {\ar "b";"d"};?*!/_2mm/{\hat{\epsilon}};
   \end{xy}
\end{array}
\]
Since~$\cM$ is a coproduct, there is a unique type-$\Delta$ map $\hat{\epsilon}$ satisfying $\hat{\epsilon}\circ\xi_i=\hat{\xi}_i$ for $i=1,2$.
Define $\tau:\cM|_X\longrightarrow \cM$ via  $x\longrightarrow x$.
From \cref{P-DelContr}(a) we know that~$\tau$ is type-$\Delta$.
Consider the map $\hat{\epsilon}\circ\tau:\cM|_X\longrightarrow\cM|_X$.
It satisfies $\hat{\epsilon}\circ\tau|_{X_i}=\id_{X_i}$ for $i=1,2$.
Hence \cref{T-SemiLinIso} implies that $\hat{\epsilon}\circ\tau:X\longrightarrow X$ is a bijective map
(and equivalent to a semi-linear isomorphism).
\\
Let now $\cN$ be any \qM{} and $\alpha_i:\cM_i\longrightarrow\cN$ be type-$\Delta$ maps.
Then there exists a type-$\Delta$ map $\epsilon:\cM\longrightarrow\cN$ resulting in the commutative diagram
\[
\begin{array}{l}
   \begin{xy}
   (-20,20)*+{\cM|_X}="e";%
   (0,40)*+{\cM_1}="a";
   (0,20)*+{\cM}="b";%
   (0,0)*+{\cM_2}="c";%
   (20,20)*+{\cN}="d";%
   {\ar "a";"b"}?*!/_-2mm/{\xi_1};
    {\ar "a";"d"}?*!/_2mm/{\alpha_1};
    {\ar "c";"b"}?*!/_2mm/{\xi_2};
   {\ar "c";"d"}?*!/_-2mm/{\alpha_2};%
   {\ar "b";"d"};?*!/_2mm/{\epsilon};
   {\ar "a";"e"}?*!/_-2mm/{\hat{\xi}_1};%
   {\ar "c";"e"}?*!/_+2mm/{\hat{\xi}_2};%
   {\ar "e";"b"};?*!/_2mm/{\tau};
   \end{xy}
\end{array}
\]
Hence the map $\gamma:=\epsilon\circ\tau:\cM|_X\longrightarrow\cN$ is type-$\Delta$ and satisfies $\gamma\circ\hat{\xi}_i=\alpha_i$.
It remains to show the uniqueness of $\gamma$.
Suppose there is also a type-$\Delta$ map $\delta:\cM|_X\longrightarrow\cN$ such that $\delta\circ\hat{\xi}_i=\alpha_i$ for $i=1,2$.
Set $\tilde{\gamma}=\gamma\circ\hat{\epsilon}\circ\tau$ and $\tilde{\delta}=\delta\circ\hat{\epsilon}\circ\tau$.
Now we have $\gamma\circ\hat{\epsilon}\circ\xi_i=\gamma\circ\hat{\xi}_i=\alpha_i$ and
$\delta\circ\hat{\epsilon}\circ\xi_i=\delta\circ\hat{\xi}_i=\alpha_i$.
Since~$\cM$ is a coproduct, we conclude that $\gamma\circ\hat{\epsilon}=\delta\circ\hat{\epsilon}$.
This implies $\tilde{\gamma}=\tilde{\delta}$ and since $\hat{\epsilon}\circ\tau$ is a bijection, this in turn yields $\delta=\gamma=\epsilon\circ\tau$.
All of this shows that $(\cM|_X,\hat{\xi}_1,\hat{\xi}_2)$ is a coproduct.
Hence~$\cM$ and $\cM|_X$ are isomorphic in \qMatd, and this means $X=\F^n$.

\underline{Claim 3:}
$\cM_1,\,\cM_2$ have a coproduct of the form $(\tilde{\cM},\iota_1,\iota_2)$, where~$\tilde{\cM}$ has ground space $\F^{n_1+n_2}$ and $\iota_i:\cM_i\longrightarrow\tilde{\cM}$ are as in~\eqref{e-iota12F}.
\\
We show first that there exists an $\cL$-isomorphism $\beta:X\longrightarrow\F^{n_1+n_2}$ such that $\beta\circ\hat{\xi}_i=\iota_i,\,i=1,2$.
In a second step this map will be turned into the desired type-$\Delta$ isomorphism between $\cM|_X$ and the new coproduct~$\tilde{\cM}$.
To show the existence of~$\beta$, use again the construction in \cref{P-DirProdMat}: consider any \qM{}~$\cN$ on $\F^{n_1+n_2}$ and the type-$\Delta$ maps
$\iota_i:\cM_i\longrightarrow\cN$ with $\iota_i$ as in \eqref{e-iota12F}.
Since $\cM|_X$ is a coproduct, there exists a type-$\Delta$ map $\beta:\cM|_X\longrightarrow\cN$ such that $\beta\circ\hat{\xi}_i=\iota_i,\,i=1,2$.
Hence $e_1,\ldots,e_n$ are in the image of~$\beta$ and thus \cref{T-SemiLinIso} implies that~$\beta$ is bijective.
All of this provides us with the desired $\cL$-isomorphism $\beta:X\longrightarrow\F^{n_1+n_2}$.
 Note that~$\beta$ is linear if $\Delta\in\{\textsf{l-s, l-rp, l-w}\}$.
Now we use \cref{R-ConstrIsom} to define a new \qM{} structure on $\F^{n_1+n_2}$.
Set $\tilde{\cM}=(\F^{n_1+n_2},\tilde{\rho})$
via $\tilde{\rho}(V):=\rho(\beta^{-1}(V))$ with~$\rho$ being the rank function of~$\cM|_X$.
This makes sense because the inverse of an $\cL$-isomorphism is again an $\cL$-map.
Thanks to \cref{R-ConstrIsom} the flats of~$\tilde{\cM}$ are given by
$\tilde{\cF}:=\{\beta(F)\mid F\in\cF(\cM|_X)\}$.
This turns~$\beta$ into an isomorphism in \qMatd{} from $\cM|_X$ to~$\tilde{\cM}$, and
the maps $\iota_i=\beta\circ\hat{\xi}_i:\cM_i\longrightarrow\tilde{\cM}$ are type-$\Delta$.
Thus $(\tilde{\cM},\beta\circ\hat{\xi}_1,\beta\circ\hat{\xi}_2)=(\tilde{\cM},\iota_1,\iota_2)$ is a coproduct, as desired.
\end{proof}

Now we are ready to show the non-existence of a coproduct in the following linear cases.

\begin{theo}\label{T-NoCoprq2}
Let $q\geq2$ and $\Delta\in\{\textsf{l-s, l-rp}\}$.
There exist representable \qM{}s  that do not have a coproduct in \qMatd.
\end{theo}

\begin{proof}
Let $\F=\F_q$ and $\cM_1=\cU_{1}(\F^2)=\cM_2$, that is~$\cM_1$ and $\cM_2$ are the uniform $q$-matroids on $\F^2$ with rank $1$.
Their collections of flats are
\begin{equation}\label{e-F1F2}
   \cF(\cM_1)=\cF(\cM_2)=\{\{0\},\F^2\}.
\end{equation}
Assume by contradiction that~$\cM_1$ and~$\cM_2$ have a coproduct.
From \cref{T-NoCopr2} we know that it is without loss of generality of the form
$(\cM,\iota_1,\iota_2)$, where~$\cM$ has ground space~$\F^4$ and $\iota_i$ are as in \eqref{e-iota12F}.
We will show that such~$\cM$ does not exist.
To do so, we construct various $q$-matroids~$\cN^{(j)}$ along with type-$\Delta$ maps $\tau_i:\cM_i\longrightarrow\cN^{(j)},\,i=1,2$.
Consider the construction in \cref{P-BlockDiag} for $m=4$.
Let $\omega\in\F_{q^4}$ be a primitive element and let~$\Omega$ and $G^{(j)},\,j\in\Omega,$ be as in that proposition.
For every $j\in\Omega$ let $\cN^{(j)}=(\F^4,\,\rho^{(j)})$ be the associated \qM{}, thus $\rho^{(j)}(\rowsp(Y))=\rk(G^{(j)}Y\T)$ for
$Y\in\F^{y\times 4}$.
Note that the uniform $q$-matroids $\cM_1=\cM_2$ are represented by every matrix $(1\ \omega^j)\in\F_{q^4}^{1\times2},j\in\Omega$.
Therefore \cref{P-DirProdMat} tells us that for every $j\in\Omega$ the maps $\iota_i:\cM_i\longrightarrow\cN^{(j)},\,i=1,2$,
are  type-$\Delta$ for either~$\Delta$ under consideration.
Since~$(\cM,\iota_1,\iota_2)$ is a coproduct in \qMatd, this implies
\begin{equation}\label{e-epsj}
  \text{for all $j\in\Omega$ there exists a type-$\Delta$ map
$\epsilon_{j}:\cM\longrightarrow\cN^{(j)}$ such that $\epsilon_{j}\circ\iota_i=\iota_i$ for $i=1,2$.}
\end{equation}
Hence $\epsilon_{j}(v_1,0)=(v_1,0)$ and $\epsilon_{j}(0,v_2)=(0,v_2)$ for all $v_1,v_2\in\F^2$.
Then linearity of~$\epsilon_{j}$ implies $\epsilon_{j}=\id_{\F^4}$ for all $j\in\Omega$ and
we arrive at the commutative diagrams
\begin{equation}\label{e-Diagrq2}
\begin{array}{l}
   \begin{xy}
   (0,40)*+{\cM_1}="a";
   (0,20)*+{\cM}="b";%
   (0,0)*+{\cM_2}="c";%
   (20,20)*+{\cN^{(j)}}="d";%
   {\ar "a";"b"}?*!/_-2mm/{\iota_1};
    {\ar "a";"d"}?*!/_2mm/{\iota_1};
    {\ar "c";"b"}?*!/_2mm/{\iota_2};
   {\ar "c";"d"}?*!/_-2mm/{\iota_2};%
   {\ar "b";"d"};?*!/_2mm/{\text{\small id}};
   \end{xy}
\end{array}
\end{equation}
where $\id:\cM\longrightarrow\cN^{(j)}$ is type-$\Delta$.
\\
a) For $\Delta=\textsf{l-rp}$ this implies that~$\cM$ is isomorphic to each~$\cN^{(j)}$, which contradicts \cref{P-BlockDiag}(d).
Hence~$\cM_1$ and~$\cM_2$ do not have a coproduct in \qMatlrp.
\\
b) Let $\Delta=\textsf{l-s}$.
Since $\id:\cM\longrightarrow\cN^{(j)}$ is a strong map for all $j\in\Omega$, we conclude that the elements of
\[
     \cF':=\bigcup_{j\in\Omega}\cF(\cN^{(j)}).
\]
have to be flats of~$\cM$.
From \cref{L-Fprime} we know that
\begin{equation}\label{e-Fprime}
  \cF'=\{0,\F^4,T_1,T_2\}\cup\{V\in\cL(\F^4)\mid \dim V\leq 2, V\cap T_1=0=V\cap T_2\},
\end{equation}
where $T_1=\subspace{1000,\,0100}$ and $T_2=\subspace{0010,\,0001}$.
Let $\cF:=\cF(\cM)$. Then $\cF$ satisfies (F1)--(F3) from \cref{T-AxFlats}.
This implies in particular that (F3) has to be true for the flats in~$\cF'$.
To investigate this further we define for $V\in\cF'$
\[
    \CovFP(V)=\{F\in \cF'\mid V< F\text{ and there is no $Z\in\cF'$ such that }V<Z<F\}.
\]
We call the elements of $\CovFP(V)$ covers of~$V$ in~$\cF'$.
Of course, the covers of $V$ in~$\cF'$ need not be covers of~$V$ in the lattice~$\cF$.
Using \eqref{e-Fprime} we can determine the covers in $\cF'$ explicitly.
For ease of notation set
\begin{align*}
   \cF_1'&=\{V\in\cF'\mid \dim V=1\}=\{V\in\cL(\F^4)\mid \dim V=1,\,V\not\subseteq T_1\cup T_2\},\\
   \cF_2'&=\{V\in\cF'\mid \dim V=2,\,T_1\neq V\neq T_2\}=\{V\in\cL(\F^4)\mid \dim V=2,\,V\cap T_1=0=V\cap T_2\}.
\end{align*}
Note that
\[
   \cF_2'=\Bigg\{\rowsp\begin{pmatrix}1&0&a&b\\0&1&c&d\end{pmatrix}\,\Bigg|\,a,b,c,d\in\F,\,
   \det\begin{pmatrix}a&b\\c&d\end{pmatrix}\neq0\Bigg\}.
\]
Then
\begin{equation}\label{e-CovFP}
   \CovFP(V)=
   \begin{cases}\{T_1,T_2\}\cup\cF_1'&\text{if }V=0,\\
         \{W\in\cF_2'\mid V\leq W\}&\text{if }V\in\cF_1',\\
         \{\F^4\}&\text{if }V\in\cF_2'\cup\{T_1,T_2\}.
    \end{cases}
\end{equation}
Note that for $V\in\cF_1'$ the set $\{W\in\cF_2'\mid V\leq W\}$ is indeed nonempty.
Choose now $V=\subspace{(1,1,1,0)}\in\cF'_1$, and $v=(1,0,0,0)$.
By property (F3) the space $V+\subspace{v}$ has to be in a unique cover, say~$\hat{F}$, of~$V$ in~$\cF$.
Since $v\in T_1$, it is clear that $\hat{F}\not\in\CovFP(V)$.
Note that $\iota_1^{-1}(\hat{F})$ contains $\iota_1^{-1}(v)=(1,0)$, and therefore $\iota_1^{-1}(\hat{F})=\F^2$
thanks to~\eqref{e-F1F2} since~$\iota_1$ is a strong map.
Hence $\hat{F}$ contains the subspace $\subspace{(1,1,1,0),\,(1,0,0,0),(0,1,0,0)}=\subspace{(0,0,1,0),\,(1,0,0,0),(0,1,0,0)}$.
Using now~$\iota_2$ we conclude that $\hat{F}=\F^4$.
But now the gradedness of the lattice~$\cF$ (see \cref{T-FlatsMatroid}(b)) together with $\cF'\subseteq\cF$ and \eqref{e-CovFP} shows that~$\F^4$ is not a cover of~$V$ in~$\cF$.
Hence we arrive at a contradiction and thus there is no coproduct of $\cM_1$ and $\cM_2$ in \qMatls
\end{proof}

With $\Delta=\textsf{l-w}$ being discussed in the next section, it remains to consider the nonlinear cases.

\begin{theo}\label{T-NoCoprq2A}
Let $q\geq 2$ and $\Delta\in\{\textsf{w, s, rp}\}$.
There exist representable \qM{}s  that do not have a coproduct in \qMatd.
\end{theo}

\begin{proof}
Let $\F=\F_q$. As in the proof of \cref{T-NoCoprq2} let $\cM_1=\cM_2=\cU_{1}(\F^2)$.
Consider the uniform \qM{} $\cN=\cU_{3}(\F^3)$.
Denote the rank functions of these $q$-matroids by $\rho_{\cM_1},\rho_{\cM_2}$, and $\rho_\cN$.
For any nonzero vector $v\in\F^2$ let $\lambda_v\in\F^*$ be its leftmost nonzero entry.
For $i=1,2$ define $\alpha_i:\F^2\longrightarrow\F^3$ via $\alpha_i(0)=0$ and
\[
   \alpha_1(v)=\lambda_v e_1 \text{ and }\alpha_2(v)=\lambda_ve_2 \ \text{ for all }v\in\F^2\setminus0,
\]
where $e_1,e_2$ are the first two standard basis vectors in~$\F^3$.
Then $\alpha_1,\alpha_2$ are $\cL$-maps.
In fact $\alpha_i(V)=\subspace{e_i}$ if $V\neq0$ and $\alpha_i(0)=0$.
We claim that $\alpha_i:\cM_i\longrightarrow\cN$ is type-$\Delta$ for each $\Delta\in\{\textsf{w, s, rp}\}$.
Indeed, every nonzero subspace $V\in\cL(\F^2)$ satisfies $\rho_{\cM_i}(V)=1=\rho_{\cN}(\alpha_i(V))$ and hence~$\alpha_i$ is rank-preserving (and thus weak).
Furthermore, note that every subspace $V\leq\F^3$ is a flat of $\cN$ and satisfies
$\alpha_i^{-1}(V)\in\{0,\,\F^2\}=\cF(\cM_i)$ (with $\alpha_i^{-1}(V)=0$ iff $e_i\not\in V$).
Thus $\alpha_i$ is strong.
\\
Now we are ready to show that $\cM_1$ and~$\cM_2$ do not have a coproduct in \qMatd.
Recall from \cref{T-NoCopr2} that if they do have a coproduct then they have one of the form $(\cM,\iota_1,\iota_2)$ with ground space~$\F^4$ and~$\iota_i$ as in~\eqref{e-iota12F}.
We will establish the non-existence of such a coproduct by showing that there is no $\cL$-map $\epsilon:\F^4\longrightarrow\F^3$ such that $\epsilon\circ\iota_i=\alpha_i$ for $i=1,2$.
\\
Assume that there does exist an $\cL$-map $\epsilon:\F^4\longrightarrow\F^3$ such that $\epsilon\circ\iota_i=\alpha_i$ for $i=1,2$.
Consider the subspaces $V=\subspace{(1,0,0,0),(0,0,1,0)}$ and  $W=\subspace{(1,0,0,0),(0,0,0,1)}$ of~$\F^4$.
Then $\epsilon(V)$ and $\epsilon(W)$ are subspaces of $\F^3$ of cardinality at most $q^2$ and contain $\subspace{e_1,e_2}$.
Hence $\epsilon(W)=\epsilon(V)=\subspace{e_1,e_2}$, and $\epsilon|_V$ and $\epsilon|_W$ are injective.
Since $e_1+e_2\in\epsilon(V)=\epsilon(W)$, there exist vectors $v\in V$ and $w\in W$ such that $\epsilon(v)=\epsilon(w)=e_1+e_2$.
These vectors must be of the form $v=\lambda(1,0,\mu,0)$ and $w=\lambda'(1,0,0,\mu')$ for some $\lambda,\lambda',\mu,\mu'\in\F^*$.
Set $U=\subspace{(1,0,\mu,0),(1,0,0,\mu')}$. Then $\epsilon(0,0,\mu,-\mu')=\alpha_2(\mu,-\mu')=\mu e_2\in\epsilon(U)$, and since
$\epsilon(U)$ is a subspace of cardinality at most $q^2$ which also contains $e_1+e_2$, we conclude that
$\epsilon(U)=\subspace{e_1,e_2}$.
Thus $\epsilon|_U$ is injective.
But this contradicts $\epsilon(v)=\epsilon(w)=e_1+e_2$.
All of this shows that there is no $\cL$-map $\epsilon$ with the desired properties and thus $\cM_1$ and $\cM_2$ do not have a coproduct in \qMatd.
\end{proof}

\section{A Coproduct in \mbox{\qMatlw}}\label{S-lw}

In this section we establish the existence of a coproduct in the category \qMatlw.
In fact, we will show that the direct sum, introduced in \cite[Sec.~7]{CeJu21} is such a coproduct.
For the sake of self-containment we will first give a concise description of the direct sum and provide the necessary proofs.
The description, different from the original one in \cite{CeJu21}, will make the presentation consistent with this paper and be helpful for proving the existence of a coproduct.
Alternative proofs of the statements below can be found at \cite[Prop.~22, Prop.~26, Def.~40]{CeJu21}.

We need to recall the notions of independent spaces and circuits of a \qM. Let $\cM=(E,\rho)$ be a \qM{}.
Then $V\in\cL(E)$ is \emph{independent} if $\rho(V)=\dim V$ and \emph{dependent} otherwise.
A dependent space all of whose proper subspaces are independent is called a \emph{circuit}.
We refer to \cite[Def.~7 and~14]{BCJ22} for the properties of the collections of independent spaces and the collection of circuits of a \qM{}.
We will only need the following two facts:
(a) every subspace of an independent space is independent;
(b) for every $V\in\cL(E)$ there exists an independent space $I\leq V$ such that $\rho(I)=\rho(V)$.

We start with the following result taken from~\cite{GLJ22Ind}.

\begin{theo}[\mbox{\cite[Thm.~3.9 and its proof]{GLJ22Ind}}]\label{T-Submodular}
Let $\tau:\cL(E)\longrightarrow\N_0$ be a map satisfying (R2), (R3) from \cref{D-qMatroid}.
Define $r_\tau:\cL(E)\longrightarrow \N_0,\ V\longmapsto \min\{\tau(X)+\dim V-\dim X\mid X\in\cL(V)\}$.
Then $r_\tau$ satisfies (R1)--(R3) and thus defines a \qM{} $\cM=(E,r_\tau)$.
Furthermore, a subspace~$V$ is independent in~$\cM$ if and only if $\tau(W)\geq\dim W$ for all $W\in V$.
\end{theo}

\begin{theo}[\mbox{\cite[Sec.~7]{CeJu21}}]\label{T-DirSum1}
Let $\cM_i=(E_i,\rho_i),\,i=1,2,$ be \qM{}s and set $E=E_1\oplus E_2$.
Let $\pi_i:E\longrightarrow E_i$ be the corresponding projection onto $E_i$.
Define $\rho'_i:\cL(E)\longrightarrow \N_0,\ V\longmapsto \rho_i(\pi_i(V))$ for $i=1,2$.
Then $\cM'_i=(E,\rho'_i)$ is a \qM{}.
Furthermore, let
\begin{equation}\label{e-rho}
  \rho:\cL(E)\longrightarrow \N_0,\ V\longmapsto \min\{\rho'_1(X)+\rho'_2(X)+\dim V-\dim X\mid X\in\cL(V)\}.
\end{equation}
Then $\cM=(E,\rho)$ is a \qM{}.
It is called the \emph{direct sum of $\cM_1$ and~$\cM_2$} and denoted by $\cM_1\oplus\cM_2$.
The collection of circuits of $\cM_1\oplus\cM_2$ is given by
\[
   \cC(\cM_1\oplus\cM_2)=\bigg\{C\in\cL(E)\,\bigg|\,
     \begin{array}{l} C\text{ is inclusion-minimal subject to the}\\ \text{condition }\rho'_1(C)+\rho'_2(C)\leq \dim C-1\end{array}\bigg\}.
\]
\end{theo}

\begin{proof}
1) We show that $\rho'_1$ is indeed a rank function. Using that $\rho_1$ is a rank function,
we have $0\leq\rho_1(\pi_1(V))\leq\dim\pi_1(V)\leq \dim V$, and this shows (R1). (R2) is trivial.
For (R3) let $V,W\in E$. Note first that $\pi_1(V\cap W)\subseteq\pi_1(V)\cap\pi_1(W)$. Thus
\begin{align*}
    \rho'_1(V+W)&=\rho_1(\pi_1(V+W))=\rho_1(\pi_1(V)+\pi_1(W))\\
           &\leq\rho_1(\pi_1(V))+\rho_1(\pi_1(W))-\rho_1(\pi_1(V)\cap\pi_1(W))\\
           &\leq\rho_1(\pi_1(V))+\rho_1(\pi_1(W))-\rho_1(\pi_1(V\cap W))=\rho'_1(V)+\rho'_1(W)-\rho'_1(V\cap W).
\end{align*}
2) For the fact that $\rho$ is a rank function consider first the function $\tau(V)=\rho'_1(V)+\rho'_2(V)$.
It is clearly non-negative and one straightforwardly
shows that it satisfies (R2), (R3) from \cref{D-qMatroid}; see also \cite[Thm.~28]{CeJu21}.
Hence \cref{T-Submodular} implies that~$\rho$ is a rank function and that the collection of circuits is as stated.
\end{proof}

The \qM{}s~$\cM_1$ and $\cM_2$ are naturally embedded in the direct sum.

\begin{theo}[\mbox{\cite[Sec.~7]{CeJu21}}]\label{T-DirSum2}
Let the data be as in \cref{T-DirSum1} and in particular $E=E_1\oplus E_2$.
Let~$\iota_i$ be as in \eqref{e-iota12}. Let $V\in\cL(E_i)$. Then
\[
   \rho'_i(\iota_i(V))=\rho_i(V)=\rho(\iota_i(V))\ \text{ and }\ \rho'_j(\iota_i(V))=0\text{ for }j\neq i.
\]
As a consequence, $\iota_i$ induces linear rank-preserving isomorphisms between~$\cM_i$ and
$\cM|_{\iota_i(E_i)}$ and $(\cM'_i)|_{\iota_i(E_i)}$.
\end{theo}

\begin{proof}
$\rho'_i(\iota_i(V))=\rho_i(V)$ follows from $\pi_i(\iota_i(V))=V$ and $\rho'_j(\iota_i(V))=0$ is clear because $\pi_j(\iota_i(V))=0$.
We show next the identity $\rho'_i(\iota_i(V))=\rho(\iota_i(V))$ for $i=1$.
Choose $X\leq \iota_1(V)$ and write $\iota_1(V)=X\oplus Y$. Using that~$\rho'_1$ is a rank function, we obtain
$\rho'_1(\iota_1(V))\leq \rho'_1(X)+\rho'_1(Y)\leq\rho'_1(X)+\dim Y=\rho'_1(X)+\dim\iota_1(V)-\dim X=\rho'_1(X)+\rho'_2(X)+\dim\iota_1(V)-\dim X$,
where the last step follows from the fact that $X\in\iota_1(E_1)$, hence $\rho'_2(X)=0$.
All of this shows that the minimum in~\eqref{e-rho} is
attained by $\rho'_1(\iota_1(V))$, and therefore $\rho'_1(\iota_1(V))=\rho(\iota_1(V))$.
The consequence is clear.
\end{proof}

Thanks to the above we may and will from now on identify subspaces~$V$ in~$E_i$ with their image~$\iota_i(V)$.
Note that by the above,~$E_i$ is a loop space of~$\cM'_j$ for $j\neq i$ (by definition a loop space is a space with rank~$0$).
The process from~$\cM_1$ to $\cM'_1$ is called \emph{adding a loop space} in \cite{CeJu21}.
It is a special instance of the direct sum because $\cM'_i\cong\cM_i\oplus\cU_0(E_j)$ for $i\neq j$.

We wish to mention that the direct sum also satisfies $(\cM_1\oplus\cM_2)/E_i\cong\cM_j$ for $j\neq i$.
This has been shown in  \cite[Cor.~48]{CeJu21} with the aid of the identity
$\rho(V_1\oplus V_2)=\rho(V_1)+\rho(V_2)$ for all subspaces $V_i\leq E_i$; see \cite[Thm.~47]{CeJu21}.

We now turn to our main result stating that $\cM_1\oplus\cM_2$ is a coproduct in \qMatlw.
We need the following lemma.

\begin{lemma}\label{T-NotWeak}
Let $\cM_i=(E_i,\rho_i),i=1,2$, be \qM{}s and $\phi:\cM_1\longrightarrow\cM_2$ be a linear map.
Suppose~$\phi$ is not a weak map.
Then there exists a circuit $C$ of~$\cM_1$ such that
\begin{equation}\label{e-Circuit}
  \rho_2(\phi(C))>\rho_1(C),\quad \dim C=\dim\phi(C),\quad \phi(C)\text{ is an independent space in~$\cM_2$.}
\end{equation}
\end{lemma}

\begin{proof}
Since~$\phi$ is not weak there exists an inclusion-minimal subspace $V\in\cL(E_1)$ such that $\rho_2(\phi(V))>\rho_1(V)$.
Clearly $V\neq0$.
We will show that~$V$ is the desired circuit and proceed in several steps.
\\
1) We first establish the following identities
\begin{equation}\label{e-VW}
  \rho_1(W)=\rho_1(V)=\rho_2(\phi(W))=\rho_2(\phi(V))-1\text{ for all $W\leq V$ with $\dim W=\dim V-1$}.
\end{equation}
Let $W\leq V$ with $\dim W=\dim V-1$. Write $V=W\oplus X$, thus~$X$ is a 1-space.
Then $\phi(V)=\phi(W)+\phi(X)$ and therefore $\dim\phi(V)\leq\dim\phi(W)+1$.
Furthermore, by minimality of~$V$ we have $\rho_2(\phi(W))\leq\rho_1(W)$.
Using the properties of rank functions, we obtain
\begin{equation}\label{e-rho12}
   \rho_2(\phi(W))\leq\rho_1(W)\leq\rho_1(V)<\rho_2(\phi(V)),
\end{equation}
and thus $\phi(W)\lneq\phi(V)$, which means $\dim\phi(V)=\dim\phi(W)+1$, and in fact
$\phi(V)=\phi(W)\oplus\phi(X)$. This implies $\rho_2(\phi(V))\leq\rho_2(\phi(W))+1$.
Together with~\eqref{e-rho12} this yields $\rho_2(\phi(V))=\rho_2(\phi(W))+1$ as well as
$\rho_2(\phi(W))=\rho_1(W)=\rho_1(V)$.
This establishes~\eqref{e-VW}.
\\
2) We show that $\phi|_V$ is injective. Assume to the contrary that there exists $v\in V\setminus 0$ such that $\phi(v)=0$.
Then $V=W\oplus\subspace{v}$ for some subspace~$W$ of~$V$ and thus $\phi(V)=\phi(W)$.
But this contradicts~\eqref{e-VW}. Hence~$\phi|_V$ is injective and $\dim V=\dim\phi(V)$.
\\
3) We show that $\phi(V)$ is independent in~$\cM_2$. Assume to the contrary that $\rho_2(\phi(V))<\dim\phi(V)$.
Then there exists an independent subspace $I\lneq\phi(V)$ such that $\rho_2(\phi(V))=\rho_2(I)=\dim I$.
Since $\phi|_V$ is injective, there exists a subspace $J\lneq V$ such that $\phi(J)=I$ and $\dim J=\dim I$.
Now we have $\dim J=\dim I=\rho_2(I)\leq\rho_1(J)\leq\dim J$, where the first inequality follows from the
minimality of~$V$ subject to $\rho_2(\phi(V))>\rho_1(V)$.
Thus we have equality throughout, which shows that~$J$ is independent in~$\cM_1$.
Furthermore, since $J\lneq V$ there exists a subspace $W\lneq V$ with $\dim W=\dim V-1$ such that $J\leq W\lneq V$.
Applying~$\phi$ we arrive at $\rho_2(I)\leq\rho_2(\phi(W))\leq \rho_2(\phi(V))=\rho_2(I)$, and we have equality throughout.
But this contradicts~\eqref{e-VW} and therefore $\phi(V)$ is independent in~$\cM_2$.
\\
4) It remains to show that~$V$ is a circuit. \eqref{e-VW} together with~2) and~3) implies that
for every hyperplane~$W$ of~$V$ we have
$\rho_1(W)=\rho_1(V)=\rho_2(\phi(V))-1=\dim\phi(V)-1=\dim V-1=\dim W$.
This shows that~$V$ is dependent and~$W$ is independent in~$\cM_1$.
Since~$W$ was an arbitrary hyperplane of~$V$, we conclude that~$V$ is a circuit.
\end{proof}

\begin{theo}\label{T-DirSumCoPr}
Let $\cM_i=(E_i,\rho_i),\,i=1,2,$ be \qM{}s and $\cM=\cM_1\oplus\cM_2=(E,\rho)$ be the direct sum as defined in \cref{T-DirSum1}.
Let $\iota_i$ be as in \eqref{e-iota12}.
Then $(\cM,\iota_1,\iota_2)$ is a coproduct of~$\cM_1$ and~$\cM_2$ in the category \qMatlw.
\end{theo}

\begin{proof}
First of all, thanks to \cref{T-DirSum2} the maps $\iota_i:\cM_i\longrightarrow\cM$ are rank-preserving, thus weak.
Let $\cN=(\tilde{E},\tilde{\rho})$ be a \qM{} and $\alpha_i:\cM_i\longrightarrow\cN$ be  linear weak maps.
We have to show the existence of a unique linear weak map $\epsilon:\cM\longrightarrow\cN$ such that $\epsilon\circ\iota_i=\alpha_i$ for $i=1,2$.
Since $\cM$ has ground space $E_1\oplus E_2$ it is clear that the only linear map satisfying $\epsilon\circ\iota_i=\alpha_i$ is given by
$\epsilon(v_1+v_2)=\alpha_1(v_1)+\alpha_2(v_2)$ for all $v_i\in E_i$; recall that we identify $E_i$ with $\iota_i(E_i)$.
Thus it remains to show that this map~$\epsilon$ is weak.
We will use \cref{T-NotWeak}.
Choose any circuit~$C$ in~$\cM$.
Denote by $\rho'_i$ and~$\pi_i$ the maps as in \cref{T-DirSum1} and set $X_i=\pi_i(C)$.
Then $C\leq X_1\oplus X_2$ and, since  $X_i\leq E_i$, we obtain
$\epsilon(C)\leq\epsilon(X_1\oplus X_2)=\epsilon(X_1)+\epsilon(X_2)=\alpha_1(X_1)+\alpha_2(X_2)$.
Applying the rank function~$\tilde{\rho}$ and using the weakness of the maps~$\alpha_i$,  we compute
\begin{align*}
    \tilde{\rho}(\epsilon(C))&\leq\tilde{\rho}(\alpha_1(X_1)+\alpha_2(X_2))\leq \tilde{\rho}(\alpha_1(X_1))+\tilde{\rho}(\alpha_2(X_2))\\
     &\leq\rho_1(X_1)+\rho_2(X_2)=\rho'_1(C)+\rho'_2(C)\leq\dim C-1,
\end{align*}
where the last step follows from  \cref{T-DirSum1} because~$C$ is a circuit.
This shows that $\epsilon(C)$ is not an independent space of~$\cN$ with the same dimension as~$C$.
Thus, no circuit in~$\cM$ satisfies~\eqref{e-Circuit}, and this shows that~$\epsilon$ is weak.
\end{proof}

\begin{exa}\label{E-UniformWeak}
Let $\F=\F_q$ and consider the \qM{}s $\cM_1=\cM_2=\cU_1(\F^2)$,
i.e., $\rho_1(V)=\rho_2(V)=\min\{1,\dim V\}$ for all $V\in\cL(\F^2)$.
The direct sum $\cM_1\oplus\cM_2$ has been determined in \cite[Ex.~48]{CeJu21} using the definition of
its rank function in~\eqref{e-rho}.
In this example we will derive the same result by making use of the fact that $\cM_1\oplus\cM_2$ is a coproduct in \qMatlw.
Let~$\omega$ be a primitive element of $\F_{q^4}$.
Then $G=( 1\ \omega^2)\in\F_{q^4}^{1\times 2}$ represents~$\cM_1=\cM_2$.
Consider $G^{(2)}$ as in \cref{P-BlockDiag} and let $\cN^{(2)}=(\F^4,\rho^{(2)})$ be the \qM{} generated by~$G^{(2)}$.
Thanks to \cref{P-BlockDiag} we have
\[
    \rho^{(2)}(V)=\min\{2,\dim V\}\text{ for }V\in\cL(\F^4)\setminus\{T_1,T_2\}\text{ and }\rho^{(2)}(T_1)=\rho^{(2)}(T_2)=1,
\]
where~$T_1,\,T_2$ are as in \cref{P-BlockDiag}.
Furthermore, \cref{P-DirProdMat} provides us with the linear rank-preserving, hence weak, maps $\iota_i:\cM_i\longrightarrow\cN^{(2)}$ for
$i=1,2$.
As a consequence, the map $\epsilon: \cM_1\oplus\cM_2\longrightarrow\cN^{(2)}$ from the previous proof is the
identity map on~$\F^4$. It thus induces a weak map from $\cM_1\oplus\cM_2$ to $\cN^{(2)}$, and this means
that the rank function $\rho$ of $\cM_1\oplus\cM_2$ satisfies
\[
   \rho(V)\geq\min\{2,\dim V\}\text{ for }V\in\cL(\F^4)\setminus\{T_1,T_2\}\text{ and }\rho(T_1)\geq1,\,\rho(T_2)\geq1.
\]
Using that $\iota_i$ are also rank-preserving maps from $\cM_i$ to $\cM_1\oplus\cM_2$, see \cref{T-DirSum2}, we obtain that
$\rho(T_1)=\rho(T_1)=1$ and $\rho(\F^4)=\rho(\iota_1(\F^2)\oplus\iota_2(\F^2))\leq \rho(\iota_1(\F^2))+\rho(\iota_2(\F^2))=
\rho_1(\F^2)+\rho_2(\F^2)=2$.
This implies that $\rho=\rho^{(2)}$ and thus $\cN^{(2)}=\cM_1\oplus\cM_2$.
\end{exa}

We close this section with the following characterization of the direct sum.

\begin{rem}\label{R-DirSumMax}
The fact that $\cM_1\oplus\cM_2$ is a coproduct in \qMatlw{} can be translated into the following
characterization of the direct sum.
For any finite-dimensional $\F$-vector space~$E$ define the set $\cS_E=\{\rho \mid\rho \text{ is a rank function on } E\}$ and the
partial order $\rho\leq\rho':\Longleftrightarrow\rho(V)\leq \rho'(V)$ for all $V\in\cL(E)$.
Let now $\cM_i=(E_i,\rho_i),\,i=1,2,$ be \qM{}s and set $E=E_1\oplus E_2$.
Then the set $\hat{\cS}=\{\rho\in\cS_E\mid \rho|_{E_i}\leq\rho_i\}$ has a unique maximum, say~$\hat{\rho}$,
and $\cM_1\oplus\cM_2=(E,\hat{\rho})$.
\end{rem}

\section{$q$-Matroids with $\cL$-Classes}\label{S-Lclass}
In this short section we discuss a different approach to maps between $q$-matroids.
Since the $q$-matroid structure is based on subspaces and different $\cL$-maps may induce the same map on subspaces (i.e., are
$\cL$-equivalent), one may define maps between $q$-matroids as maps between subspace lattices induced by $\cL$-maps.
Precisely, for an $\cL$-map $\phi:E_1\longrightarrow E_2$ define the \emph{$\cL$-class} as
$[\phi]=\{\psi:E_1\longrightarrow E_2\mid \psi\sim_\cL\phi\}$.
Then
\begin{equation}\label{e-phiclass}
      [\phi]:\cL(E_1)\longrightarrow\cL(E_2),\,V\longmapsto \phi(V),
\end{equation}
is well-defined (and equals $\phi_\cL$).
Since the type of a map (see \cref{D-StrMap}) is invariant under $\cL$-equivalence, this gives rise to strong, weak, and rank-preserving $\cL$-classes between $q$-matroids.
Furthermore, we call an $\cL$-class \emph{linear}, if it contains a linear $\cL$-map (not all maps in a linear $\cL$-class are linear since by \cref{P-S1dim} we may tweak a linear map into an $\cL$-equivalent nonlinear $\cL$-map).
Setting $[\phi_1]\circ[\phi_2]:=[\phi_1\circ\phi_2]$, which is indeed well-defined,
we obtain categories
\[
       \mbox{\qMat}^{[\Delta]}\text{ for }\Delta\in\{\textsf{s, rp, w, l-s, l-rp, l-w}\},
\]
in which the morphisms are $\cL$-classes of the specified type.
Being isomorphic or linearly isomorphic in the sense of \cref{D-IsomMatr} does not change when moving to $\cL$-classes.

It turns out that $\cL$-classes are not better behaved than $\cL$-maps.
The following result suggests that in fact $\cL$-maps, rather than $\cL$-classes, are the appropriate notion of maps between $q$-matroids.

\begin{theo}\label{T-NonCoprLclasses}
None of the categories  \qMat$^{[\Delta]},\,\Delta\in\{\textsf{s, rp, w, l-s, l-rp, l-w}\},$ has coproduct except for the case
$(q,[\Delta])=(2,[\textsf{l-w}])$.
\end{theo}

The proof is for the most part straightforward, but tedious. We provide a sketch.

\textit{Sketch of Proof.}
In essence one follows the proofs of Sections \ref{S-NoCopr} and \ref{S-lw} and replaces $\cL$-maps by their $\cL$-classes.
Consequently, equality of maps (which is pointwise) needs to be replaced by equality of $\cL$-classes on all 1-spaces
(see \cref{P-S1dim}(a)).
With this in mind, one verifies the following.
\\[.5ex]
1) The proof of \cref{T-NoCopr2} reducing coproducts to the form $(\cM,[\iota_1],[\iota_2])$ carries through without additional changes.
\\[.5ex]
2) $\Delta\in\{\textsf{s, rp, w}\}$. The proof of \cref{T-NoCoprq2A} also generalizes without additional changes.
\\[.5ex]
3) $\Delta\in\{\textsf{l-s, l-rp}\}$. The proof of \cref{T-NoCoprq2} needs a bit more care.
With the strategy described above we arrive at statement~\eqref{e-epsj}, which now reads as
\[
   \text{for all $j\in\Omega$ there exists a type-$[\Delta]$ map $[\epsilon_j]$ such that $[\epsilon_{j}\circ\iota_i]=[\iota_i]$ for $i=1,2$.}
\]
Thus $[\epsilon_{j}](\subspace{(v_1,0)})=\subspace{(v_1,0)}$ and $[\epsilon_{j}](\subspace{(0,v_2)})=\subspace{(0,v_2)}$
for all $v_1,v_2\in\F^2$ and all $j\in\Omega$.
Since the $\cL$-class $[\epsilon_{j}]$ is linear, \cref{P-EquivMaps}(e) implies that, without loss of generality, there exist $a_j\in\F^*$ such that
\[
   \epsilon_{j}(x_1,x_2,x_3,x_4)=(x_1,x_2,x_3,x_4)\diag(1,1,a_j,a_j)
    \text{ for all }(x_1,x_2,x_3,x_4)\in\F^4,
\]
where $\diag(1,1,a_j,a_j)$ is the $4\times4$-diagonal matrix with the specified diagonal entries.
In particular,~$\epsilon_j$ need not be the identity and in Diagram~\eqref{e-Diagrq2} we have potentially different maps
$[\epsilon_{j}]:\cM\longrightarrow\cN^{(j)}$ for each $j\in\Omega$.
Note that each~$\epsilon_j$ is an $\cL$-isomorphism.
Now we consider the two types of \cref{T-NoCoprq2}.
\\
a) Let $\Delta=\textsf{l-rp}$.
Then~$\cM$ is linearly isomorphic to each $\cN^{(j)}$, contradicting \cref{P-BlockDiag}(d).
\\
b) Let $\Delta=\textsf{l-s}$.
Consider the $q$-matroids $\cN^{(j)}=(\F^4,\rho^{(j)})$.
We claim that
$\rho^{(j)}(V)=\rho^{(j)}([\epsilon_j](V))$.
Indeed, let $V=\rowsp(Y)$ for some $Y\in\F^{y\times 4}$.
Then $\epsilon^{(j)}(V)=\rowsp(Y\diag(1,1,a_j,a_j))$.
The definition of $G^{(j)}$ implies that $G^{(j)}\diag(1,1,a_j,a_j)=\diag(1,a_j)G^{(j)}$, and thus
\[
   \rho^{(j)}(\epsilon_j(V))=\rk\big(G^{(j)}\diag(1,1,a_j,a_j)Y\T\big)
         =\rk\big(\diag(1,a_j)G^{(j)}Y\T\big)=\rk\big(G^{(j)}Y\T\big)=\rho^{(j)}(V).
\]
With the aid of \eqref{e-FPrime} we conclude that $[\epsilon_j](\cF(\cN^{(j)}))=\cF(\cN^{(j)})$ for all $j\in\Omega$.
Now we are ready to return to the proof of \cref{T-NoCoprq2} and specifically to the set
$ \cF'=\bigcup_{j\in\Omega}\cF(\cN^{(j)})$.
We claim that $\cF'\subseteq\cF(\cM)$.
Indeed, let $F\in\cF'$.
Then $F\in\cF(\cN^{(j)})$ for some $j\in\Omega$ and thus $[\epsilon_j](F)\in\cF(\cN^{(j)})$.
Since $[\epsilon_j]$ is strong, we conclude that $F=[\epsilon_j^{-1}]([\epsilon_j](F))$ is a flat in~$\cM$.
Now the rest of the proof of \cref{T-NoCoprq2}, starting at~\eqref{e-Fprime}, generalizes to $\cL$-classes as described above, and
leads to a contradiction.
\\[.5ex]
4) $(q,\Delta)=(2,\textsf{l-w})$.
Note that for $q=2$ linear maps are $\cL$-equivalent if and only if they are equal (see \cref{P-EquivMaps}(e)).
Therefore the result from \cref{T-DirSumCoPr} remains valid for \qMat$^{[\textsf{l-w}]}$ if $q=2$.
\\[.5ex]
5) $q>2$ and $\Delta=\textsf{l-w}$.
Following the proof of  \cref{T-DirSumCoPr}, we see that the existence of a coproduct implies the existence of a unique $\cL$-class
$[\epsilon]$ satisfying $[\epsilon]\circ[\iota_i]=[\alpha_i]$ for $i=1,2$.
However, choosing the linear weak map~$\epsilon$ as in that proof, we can now take any linear map
$\epsilon'$ such that  $\epsilon'|_{E_1}=\lambda_1\epsilon|_{E_1}$ and $\epsilon'|_{E_2}=\lambda_2\epsilon|_{E_2}$ for some
$\lambda_1,\lambda_2\in\F^*$ and obtain a linear weak map satisfying $[\epsilon']\circ[\iota_i]=[\alpha_i]$ for $i=1,2$.
Choosing $\lambda_1\neq\lambda_2$ we obtain $[\epsilon']\neq[\epsilon]$, which proves that the uniqueness of the map~$[\epsilon]$
fails. Therefore $\cM_1\oplus\cM_2$ is not a coproduct of~$\cM_1$
and~$\cM_2$ in \qMat$^{[\textsf{l-w}]}$.
A  similar reasoning shows that the two \qM{}s do not have any coproduct in \qMat$^{[\textsf{l-w}]}$.
\mbox{}\hfill$\square$

\section*{Open Problems}
\begin{arabiclist}
\item In \cite{GLJ22Rep} it was shown that the direct sum of two representable \qM{}s $\cM_1$ and~$\cM_2$ is not necessarily representable.
However, if $\cM=\cM_1\oplus\cM_2$ is representable, then it must represented via a block diagonal matrix $G := \text{diag}(G_1, G_2)$, where
$G_i$ represents~$\cM_i$. Establishing representability of \qM{}s is a notoriously difficult problem.
It is therefore interesting to ask if the property of the direct sum established in \cref{R-DirSumMax} can help understand when a direct sum is representable?
\item Of course, many other categorical questions can be posed, such as the existence of products, equalizers, pullbacks and pushouts.
        The article \cite{HePa18}, which discusses properties of the category of matroids with strong maps, appears to be a very useful guideline.
        Considering the results of this paper, one can expect further significant differences to the case of classical matroids.
\end{arabiclist}

\section*{Acknowledgement}
Heide Gluesing-Luerssen was partially supported by the grant \#422479 from the Simons Foundation.
We wish to thank the referees for their careful reading and constructive comments.

\bibliographystyle{abbrv}

\begin{thebibliography}{10}

\bibitem{Art57}
E.~Artin.
\newblock {\em Geometric Algebra}.
\newblock Interscience Publishers, Inc., New York, 1957.

\bibitem{BCIJS23}
E.~Byrne, M.~Ceria, S.~Ionica, R.~Jurrius, and E.~Sa{\c{c}}ikara.
\newblock Constructions of new matroids and designs over {$\text{GF}(q)$}.
\newblock {\em Des. Codes Cryptogr.}, 91:451--473, 2023.

\bibitem{BCJ22}
E.~Byrne, M.~Ceria, and R.~Jurrius.
\newblock Constructions of new $q$-cryptomorphisms.
\newblock {\em J. Comb. Theory. Ser. B}, 153:149--194, 2022.

\bibitem{CeJu21}
M.~Ceria and R.~Jurrius.
\newblock The direct sum of \mbox{$q$}-matroids.
\newblock Preprint 2021. arXiv: 2109.13637v4.

\bibitem{CrRo70}
H.~Crapo and G.-C. Rota.
\newblock {\em On the Foundations of Combinatorial Theory: {C}ombinatorial Geo-
  metries}.
\newblock MIT Press, 1970.

\bibitem{DaPr02}
B.~A. Davey and H.~A. Priestley.
\newblock {\em Introduction to Lattices and Order}.
\newblock Cambridge University Press, Cambridge, 2nd edition, 2002.

\bibitem{GLJ22Ind}
H.~Gluesing-Luerssen and B.~Jany.
\newblock Independent spaces of $q$-polymatroids.
\newblock {\em Algebr. Comb.}, 5:727--744, 2022.

\bibitem{GLJ22Gen}
H.~Gluesing-Luerssen and B.~Jany.
\newblock $q$-{P}olymatroids and their relation to rank-metric codes.
\newblock {\em J. Algebr. Comb.}, 56:725--753, 2022.

\bibitem{GLJ22Rep}
H.~Gluesing-Luerssen and B.~Jany.
\newblock Representability of the direct sum of $q$-matroids.
\newblock arXiv: 2211.11626, 2022.

\bibitem{GLJ23DSCyc}
H.~Gluesing-Luerssen and B.~Jany.
\newblock Decompositions of $q$-matroids using cyclic flats.
\newblock arXiv: 2302.02260, 2023.

\bibitem{GJLR19}
E.~Gorla, R.~Jurrius, H.~L{\'o}pez, and A.~Ravagnani.
\newblock Rank-metric codes and $q$-polymatroids.
\newblock {\em J. Algebraic Combin.}, 52:1--19, 2020.

\bibitem{Har08}
R.~Hartshorne.
\newblock Publication history of von {S}taudt's {G}eometrie der {L}age.
\newblock {\em Arch. Hist. Exact Sci.}, 62:297--299, 2008.

\bibitem{HePa18}
C.~Heunen and V.~Patta.
\newblock The category of matroids.
\newblock {\em Appl. Categ. Struct.}, 26:205--237, 2018.

\bibitem{Ja22}
B.~Jany.
\newblock The projectivization matroid of a $q$-matroid.
\newblock arXiv: 2204.01232, 2022.

\bibitem{JuPe18}
R.~Jurrius and R.~Pellikaan.
\newblock Defining the $q$-analogue of a matroid.
\newblock {\em Electron. J. Combin.}, 25:P3.2, 2018.

\bibitem{Ox11}
J.~Oxley.
\newblock {\em Matroid Theory}.
\newblock Oxford Graduate Text in Mathematics. Oxford University Press, 2nd
  edition, 2011.

\bibitem{Put}
A.~Putman.
\newblock The fundamental theorem of projective geometry.
\newblock Available at www3.nd.edu/ $\tilde{\;}$andyp/notes/FunThmProjGeom.pdf.

\bibitem{Rom08}
S.~Roman.
\newblock {\em Lattices and Ordered Sets}.
\newblock Springer, 2008.

\bibitem{Shi19}
K.~Shiromoto.
\newblock Codes with the rank metric and matroids.
\newblock {\em Des. Codes Cryptogr.}, 87:1765--1776, 2019.

\bibitem{Wh86}
N.~White.
\newblock {\em Theory of Matroids}.
\newblock Cambridge University Press, Cambridge, 1986.

\end{thebibliography}

\end{document}